\documentclass[11pt]{amsart}

\usepackage{xypic,amsmath,amssymb}

\newtheorem{proposition}{Proposition}[section]
\newtheorem{lemma}[proposition]{Lemma}

\newtheorem{theorem}[proposition]{Theorem}

\theoremstyle{definition}
\newtheorem{definition}[proposition]{Definition}
\newtheorem{example}[proposition]{Example}

\theoremstyle{remark}
\newtheorem{remark}[proposition]{Remark}

\newcommand{\thlabel}[1]{\label{th:#1}}
\newcommand{\thref}[1]{Theorem~\ref{th:#1}}
\newcommand{\selabel}[1]{\label{se:#1}}
\newcommand{\seref}[1]{Section~\ref{se:#1}}
\newcommand{\lelabel}[1]{\label{le:#1}}
\newcommand{\leref}[1]{Lemma~\ref{le:#1}}
\newcommand{\prlabel}[1]{\label{pr:#1}}
\newcommand{\prref}[1]{Proposition~\ref{pr:#1}}

\newcommand{\relabel}[1]{\label{re:#1}}
\newcommand{\reref}[1]{Remark~\ref{re:#1}}
\newcommand{\exlabel}[1]{\label{ex:#1}}
\newcommand{\exref}[1]{Example~\ref{ex:#1}}
\newcommand{\delabel}[1]{\label{de:#1}}

\newcommand{\eqlabel}[1]{\label{eq:#1}}
\newcommand{\equref}[1]{(\ref{eq:#1})}

\def\mapright#1{\smash{\mathop{\longrightarrow}\limits^{#1}}}

\def\doublerightbis#1#2{{\lower.2ex\vbox{
\hbox{${\smash{\mathop{\longrightarrow}\limits^{#1}}}$}\vspace*{-4mm}
\hbox{${\smash{\mathop{\longrightarrow}\limits_{#2}}}$}}}}

\def\equal#1{\smash{\mathop{=}\limits^{#1}}}
\def\congo#1{\smash{\mathop{\cong}\limits^{#1}}}

\newcommand{\can}{{\rm can}}

\newcommand{\Hom}{{\rm Hom}}

\newcommand{\End}{{\rm End}}

\newcommand{\Ker}{{\rm Ker}}

\newcommand{\im}{{\rm Im}}

\def\Desc{\underline{\rm Desc}}
\def\Cog{\underline{\Cc}}

\def\lan{\langle}
\def\ran{\rangle}
\def\ot{\otimes}

\def\rightact{\hbox{$\leftharpoonup$}}
\def\leftact{\hbox{$\rightharpoonup$}}

\newcommand{\Aa}{\mathcal{A}}
\newcommand{\Bb}{\mathcal{B}}
\newcommand{\Cc}{\mathcal{C}}
\newcommand{\Dd}{\mathcal{D}}

\newcommand{\Mm}{\mathcal{M}}

\newcommand{\Vv}{\mathcal{V}}

\def\*C{{}^*\hspace*{-1pt}{\Cc}}

\def\text#1{{\rm {\rm #1}}}

\def\ol{\overline}
\def\ul{\underline}

\allowdisplaybreaks[4]

\begin{document}
\title[Descent and Galois theory]{Descent and Galois theory for Hopf categories}
\author{S. Caenepeel}
\address{Faculty of Engineering,
Vrije Universiteit Brussel, Pleinlaan 2, B-1050 Brussels, Belgium}
\email{scaenepe@vub.ac.be}
\urladdr{http://homepages.vub.ac.be/\~{}scaenepe/}
\author{T. Fieremans}
\address{Faculty of Engineering,
Vrije Universiteit Brussel, Pleinlaan 2, B-1050 Brussels, Belgium}
\email{tfierema@vub.ac.be}
\urladdr{http://homepages.vub.ac.be/\~{}tfierema/}



\begin{abstract}
Descent theory for linear categories is developed. Given a linear category as an extension
of a diagonal category, we introduce descent data, and the category of descent data is
isomorphic to the category of representations of the diagonal category, if some flatness
assumptions are satisfied. Then Hopf-Galois descent theory for linear Hopf categories,
the Hopf algebra version of a linear category, is developed. This leads to the notion
of Hopf-Galois category extension. We have a dual theory, where actions by dual linear
Hopf categories on linear categories are considered.  Hopf-Galois category extensions
over groupoid algebras correspond to strongly graded linear categories.
\end{abstract}
\maketitle

\section*{Introduction}\selabel{0}
A $k$-linear category is a category enriched in the monoidal category of vector spaces $\Mm_k$.
It is a generalization of a $k$-algebra in the sense that a $k$-linear category with one object
is simply a $k$-algebra. Thus we can regard a $k$-linear category as a multi-object version
of a $k$-algebra. This philosophy was further examined in \cite{BCV}, leading to multi-object
versions of bialgebras and Hopf algebras, respectively termed $k$-linear semi-Hopf categories
and $k$-linear Hopf categories. It turns out that several classical properties of Hopf algebras
can be generalized to Hopf categories, see \cite{BCV} for some examples. One of the results
in \cite{BCV} is the fundamental theorem for Hopf modules, opening the way to Hopf-Galois theory.
The main aim of this paper is to develop Hopf-Galois theory for Hopf categories.\\
Hopf-Galois objects were introduced by Chase and Sweedler \cite{CS}, and was generalized
by Kreimer and Takeuchi \cite{KT}. One of the important properties is the Fundamental Theorem,
which can be interpreted as Hopf-Galois descent, and can be stated as follows: if $A$ is an $H$-comodule
algebra, with coinvariant subalgebra $B$, then there is an adjunction between $B$-modules and
relative Hopf modules, which is a pair of inverse equivalences if $A$ is a Hopf-Galois extension of $B$,
which is faithfully flat as a left $B$-module. From the point of view of descent theory: an $A$-module can
be descended to a $B$-module if it has the additional structure of a relative Hopf module, in other words,
the relative Hopf modules become the Hopf-Galois descent data. An elegant formulation of the theory
was given by Brzezi\'nski in \cite{Brzezinski02}, based on the theory of corings. In this formalism,
classical descent data (as introduced in \cite{Gr} for schemes, in \cite{KO} for extensions of commutative
rings, and in \cite{Cipolla} for extensions of non-commutative rings) as well as Hopf-Galois descent data
become comodules over certain corings. A more detailed account of this approach is presented in the
survey paper \cite{C}, which is at the basis of the methods developed in this paper, with one important
drawback, namely the fact that, as far as we could figure it out, the formulation in terms of corings is 
not working in the setting of Hopf categories. However, the general philosophy survives, and enables us
to formulate faithfully flat descent and Hopf-Galois theory for Hopf categories.\\
The line-up of the paper is as follows. Preliminary results from \cite{BCV} are given in \seref{1}.
In \seref{2}, we present faithfully flat descent theory for linear categories. In the classical theory, both the
base $B$ and the extension $A$ are algebras, connected by an algebra morphism. In our setting, $B$ and
$A$ are linear categories, connected by a so-called extension. But now $B$ is a diagonal category, meaning
that $B_{xy}=0$ if $x\neq y$.  Another difference, already mentioned above, is that the descent data
cannot be interpreted as comodules over a coring. In \seref{3}, we generalize the notions of comodule algebra
and its coinvariants
and relative Hopf module. The strategy to develop Hopf-Galois descent is now the following. There is a
functor from descent data to relative Hopf modules, see \prref{3.4}. A Hopf category is called an $H$-Galois
category extension of its coinvariants if a collection of canonical maps is invertible; some equivalent conditions
are given in \thref{3.5} and in this case descent data and relative Hopf modules are isomorphic categories,
leading to the desired descent theory if some flatness conditions are satisfied. In the classical case, an
alternative description of descent data is possible if $A$ is finitely generated projective as a $B$-module:
it is the category of modules over ${}_B\End(A)$. This is generalized to the categorical situation in \seref{5}.
It turns out that we need clusters of $k$-linear categories, these are collections of $k$-linear categories
indexed by $X$, see \seref{4}. Some duality results are discussed in Sections \ref{se:6} and \ref{se:7}. In the final
\seref{8}, we focus on Galois category extensions over a Hopf category induced by a groupoid, and link our
results to the work of Lundstr\"om \cite{L}, generalizing an old result of Ulbrich \cite{U} that a Hopf-Galois
extension over a group algebra is a strongly graded ring. Another classical result, already observed in
\cite{CS}, is that classical Galois extensions, where a finite group $G$ acts on $A$, are precisely Hopf-Galois
extensions over the dual of the group ring $(kG)^*$. Although we have a duality theory, see Sections \ref{se:6} and \ref{se:7},
this result cannot be generalized in a satisfactory way at the moment, we refer to the final remark \reref{8.3}
for full explanation. This will be the topic of a forthcoming paper.
\pagebreak

\section{Preliminary results}\selabel{1}
\subsection{$k$-linear categories}\selabel{1.1}
Let $\Vv$ be a monoidal category. From \cite[Sec. 6.2]{B}, we recall the notion of
$\Vv$-category. In particular, in the case where $\Vv=\Mm_k$, the category of vector spaces
over a field $k$ (or, more generally, the category of modules over a commutative ring $k$),
a $\Vv$-category is a $k$-linear category. In \cite{BCV}, the notions of $\Cc(\Vv)$-category and
Hopf $\Vv$-category are introduced. In this paper, we will work over $\Vv=\Mm_k$
and over $\Vv=\Mm_k^{\rm op}$. A Hopf $\Mm_k$-category is called a $k$-linear
Hopf category, while a Hopf $\Mm_k^{\rm op}$-category is called a dual $k$-linear
Hopf category. Let us
specify the definitions from \cite{BCV} to this particular situation.\\
Let $A$ be a $k$-linear category, and let $X$ be the class of objects in $A$. For $x,y\in A$, we
write $A_{xy}$ for the $k$-module of morphisms from $y$ to $x$. For all $x,y,z\in X$, we then
have the composition maps
$m_{xyz}:\ A_{xy}\ot A_{yz}\to A_{xz}$, $m_{xyz}(a\ot b)=ab$, for all $a\in A_{xy}$ and
$b\in A_{yz}$. The unit element of $A_{xx}$ is denoted by $1_x$.\\
Let $A$ and $A'$ be $k$-linear categories with the same underlying class of objects $X$.
A $k$-linear functor $f:\ A\to B$ that is the identity on $X$ is called a $k$-linear $X$-functor:
for all $x,y\in X$, $f_{xy}:\ A_{xy}\to B_{xy}$ is a $k$-linear map preserving multiplication and
unit.\\
For a class $X$, we introduce the category $\Mm_k(X)$.  An object is a family of objects $M$ in $\Mm_k$ indexed by $X\times X$:
$$M=(M_{xy})_{x,y\in X}.$$
A morphism $\varphi:\ M\to N$ consists of a family of $k$-linear maps $\varphi_{xy}:\ M_{xy}\to N_{xy}$
indexed by $X\times X$.\\
Let $A$ be a $k$-linear category. A right $A$-module is an object $M$ in $\Mm_k(X)$ together with a family of $k$-linear maps
$$\psi=
\psi_{xyz}:\ M_{xy}\ot A_{yz}\to M_{xz},~~\psi_{xyz}(m\ot a)=ma$$
such that the following associativity and unit conditions hold:
$(ma)b=m(ab)$; $m1_y=m$,
for all $m\in M_{xy}$, $a\in A_{yz}$ and $b\in A_{zu}$.\\
Let $M$ and $N$ be right $A$-modules. A morphism $\varphi:\ M\to N$ in $\Mm_k(X)$ is called right $A$-linear
if
$\varphi_{xz}(ma)=\varphi_{xy}(m)a$, for all $m\in M_{xy}$ and $a\in A_{yz}$. The category of right $A$-modules
and right $A$-linear morphisms is denoted by $\Mm_k(X)_A$.\\
We will also need the category $\Dd_k(X)$. Objects are families of $k$-modules
$N=(N_x)_{x\in X}$ indexed by $X$, and a morphism $N\to N'$ consists of a
family of $k$-linear maps $N_x\to N'_x$. $\Dd_k(X)$ is a symmetric monoidal category, and
an algebra $B$ in $\Dd_k(X)$ consists of a family of $k$-algebras $B=(B_x)_{x\in X}$ indexed by $X$.
We can consider $B$ as a $k$-linear category: $B_{xy}=\{0\}$ if $x\neq y$ and $B_{xx}=B_x$. $B$ is then called
a diagonal $k$-linear category. A diagonal right $B$-module is an object $N\in \Dd_k(X)$ such that every
$N_x$ is a right $B_x$-module. $\Dd_k(X)_B$ is the category of diagonal $B$-modules and right $B$-linear morphisms.
A morphism $g$ in $\Dd_x$ between right $B$-modules is right $B$-linear if every $g_x$ is right $B_x$-linear.
The category of left $B$-modules is defined in a similar way.

\subsection{Finitely generated projective modules and dual basis}\selabel{1.2}
Let $B$ be a $k$-algebra, and assume that $M$ is a finitely generated projective left $B$-module. Then
$M^*={}_B\Hom(M,B)$ is a finitely generated projective right $B$-module, with action given by the
formula
$(m^*\cdot b)(m)=m^*(m)b$,
for all $m\in M$, $m^*\in M^*$ and $b\in B$. $M$ has a finite dual basis $\sum_i e_i^*\ot_B e_i\in M^*\ot_B M$
satisfying the formulas
\begin{equation}\eqlabel{1.2.1}
\sum_i e_i^*(m) e_i= m~~~{\rm and}~~~\sum_i e_i^*m^*(e_i)=m^*,
\end{equation}
for all $m\in M$ and $m^*\in M^*$. For $N\in {}_B\Mm$ and $P\in \Mm_{B}$, we have isomorphisms
\begin{equation}\eqlabel{1.2.2}
{}_B\Hom(M,N)\cong M^*\ot_B N~~~{\rm and}~~~\Hom_B(M^*,P)\cong P\ot_B M.
\end{equation}
For later use, we provide the explicit description of $\alpha:\ P\ot_B M\to \Hom_B(M^*,P)$
and its inverse. For $p\in P$, $m\in M$, $m^*\in M^*$ and $f\in \Hom_B(M^*,P)$, we have
\begin{equation}\eqlabel{1.2.3}
\alpha(p\ot_B m)(m^*)=pm^*(m)~~{\rm and}~~\alpha^{-1}(f)=\sum_i f(e_i^*)\ot_B e_i.
\end{equation}
A left $B$-progenerator (in the literature also termed as a faithfully projective left $B$-module)
is a finitely generated projective left $B$-module that is also a generator, that is,
${\rm Tr}(M)=\{\sum_i m_i^*(m_i)~|~m_i\in M, ~m^*_i\in M^*\}=B$. A finitely generated projective
module is flat, and a progenerator is faithfully flat.\\
Let $B$ be a diagonal $k$-linear category. We can view $B$ as a $k$-linear category, see \seref{1.1}.
Consider a left $B$-module $M$. We introduce the following terminology.\\
$M$ is called locally flat, resp. locally finite as a left $B$-module if every $M_{xy}$ is flat, resp.
finitely generated projective as a left $B_x$-module.\\
A locally flat left $B$-module $M$ is called locally faithfully flat if every $M_{xx}$ is
faitfhully flat as a left $B_x$-module; A locally finite left $B$-module $M$ is called locally faithfully projective
 if every $M_{xx}$ is faitfhully flat as a left $B_x$-module.

\subsection{Hopf categories}\selabel{1.3}
The category $\Cog(\Mm_k)$ of $k$-coalgebras is a monoidal category, so we can
consider $\Cog(\Mm_k)$-categories. It is shown in \cite{BCV} that a $\Cog(\Mm_k)$-category
is a $k$-linear category $H$ with the following additonal structure: for all $x,y\in X$, $H_{xy}$ is a $k$-coalgebra
with structure maps $\Delta_{xy}$ and $\varepsilon_{xy}$ such that the following properties
hold, for all $h\in H_{xy}$ and $k\in H_{yz}$:
\begin{eqnarray*}
&&\Delta_{xz}(hk)=h_{(1)}k_{(1)}\ot h_{(2)}k_{(2)}~~~;~~~
\Delta_{xx}(1_x)=1_x\ot 1_x;\\
&&\varepsilon_{xz}(hk)=\varepsilon_{xy}(h)\varepsilon_{yz}(k)~~~;~~~
\varepsilon_{xx}(1_x)=1.
\end{eqnarray*}
A $\Cog(\Mm_k)$-category with one object is a bialgebra; an obvious name for $\Cog(\Mm_k)$-categories
in general therefore seems to be ``$k$-linear bicategories''. However, this terminology is badly chosen,
because of possible confusion with the existing notions of 2-categories and bicategories. This is
why we introduce the name ``$k$-linear semi-Hopf categories'' for $\Cog(\Mm_k)$-categories.\\
A $k$-linear Hopf category is a $\Cog(\Mm_k)$-category together with $k$-linear maps
$S_{xy}:\ H_{xy}\to H_{yx}$ such that
\begin{equation}\eqlabel{1.1.1}
h_{(1)}S_{xy}(h_{(2)})=\varepsilon_{xy}(h)1_x~~;~~S_{xy}(h_{(1)})h_{(2)}=\varepsilon_{xy}(h)1_y,
\end{equation}
for all $x,y\in X$ and $h\in H_{xy}$.\\

The same construction can be performed in the opposite category of vector spaces, leading
to the following notions.
A dual $k$-linear semi-Hopf category $K$ consists of a an object $K\in \Mm_k(X)$, together with
comultiplication and counit maps
$\Delta_{xyz}:\ K_{xz}\to K_{xy}\ot K_{yz}$ and $\varepsilon_x:\ K_{xx}\to k$,
satisfying the obvious coassociativity and counit properties. We adopt the Sweedler notation
$\Delta_{xyz}(k)=k_{(1,x,y)}\ot k_{(2,y,z)}$,
for $k\in K_{xz}$. Furthermore, every $K_{xy}$ is a $k$-algebra, with unit $1_{xy}$. The
compatibility relations between the two structures are the following, for
$h,k\in K_{xz}$ and $l,m\in K_x$:
\begin{eqnarray*}
&&\Delta_{xyz}(hk)= h_{(1,x,y)}k_{(1,x,y)}\ot h_{(2,y,z)}k_{(2,y,z)}~~~;~~~
\Delta_{xyz}(1_{xz})= 1_{xy}\ot 1_{yz};\\
&&\varepsilon_x(lm)=\varepsilon_x(l)\varepsilon_x(m)~~~;~~~
\varepsilon_x(1_{xx})=1.
\end{eqnarray*}
A dual $k$-linear Hopf category is a dual $k$-linear semi-Hopf category $K$ with an antipode $T$
consisting of a family of maps $T_{xy}:\ K_{yx}\to K_{xy}$ satisfying the following
equations, for all $l\in K_{xx}$ and $x\in X$:
$$l_{(1,x,y)}T_{xy}(l_{(2,y,x)})=\varepsilon_x(l)1_{xy}~~{\rm and}~~
T_{yx}(l_{(1,x,y)})l_{(2,y,x)}=\varepsilon_x(l)1_{yx}.$$
From \cite[Theorem 5.6]{BCV}, we recall that there is a duality between the categories of locally finite (semi-)Hopf categories and locally finite dual (semi-)Hopf categories.\\
For later use, we give the explicit description of the dual $k$-linear Hopf category $K$ corresponding to
a locally finite
$k$-linear Hopf category $H$. As an object
of $\Mm_k(X)$, $K$ is given componentwise as
$K_{xy}=H^*_{yx}$.
The multiplication on $K_{xy}$ is given by opposite convolution:
$$\lan kl,h\ran=\lan k, h_{(2)}\ran \lan l, h_{(1)}\ran,$$
for all $k,l\in K_{xy}$ and $h\in H_{yx}$. The unit of $K_{xy}$ is $1_{xy}=\varepsilon_{yx}$. The comultiplication maps
$$\Delta_{xyz}:\ K_{xz}\to K_{xy}\ot K_{yz},~~\Delta_{xyz}(k)= k_{(1,x,y)}\ot k_{(2,y,z)}$$
are characterized by the formulas
$$\lan k_{(1,x,y))}, h'\ran \lan k_{(2,y,z))}, h\ran=\lan k, hh'\ran,$$
for all $h\in H_{zy}$ and $h'\in H_{yx}$. The counit maps $\varepsilon_x:\ K_{xx}\to k$ are given by
$\varepsilon_x(k)=\lan k,1_x\ran$. The antipode maps are $T_{yx}=S_{xy}^*:\ K_{xy}=H^*_{yx}\to K_{yx}=H^*_{xy}$.

\section{Descent theory for $k$-linear categories}\selabel{2}
Let $A$ and $B$ be $k$-linear categories, with underlying class $X$, and assume that $B$ is diagonal.
Let $i:\ B\to A$ be a $k$-linear $X$-functor. Then $B$ consists of a family of $k$-algebras indexed by $X$,
and for every $x\in X$, we have a $k$-algebra morphism $i_x:\ B_x\to A_{xx}$.

\begin{definition}\delabel{2.1}
A descent datum $(M,\sigma)$ for the functor $i$ consists of a right $A$-module $M$ together with a
family of $k$-linear maps $\sigma=(\sigma_{xy})_{x,y\in X}$, where
$$\sigma_{xy}:\ M_{xy}\to M_{xx}\ot_{B_x} A_{xy}.$$
We use the following Sweedler-type notation: $\sigma_{xy}(m)=m_{<0>}\ot_{B_x} m_{<1>}$, for
$m\in M_{xy}$. The following conditions have to be satisfied, for all $m\in M_{xy}$ and $a\in A_{yz}$:
\begin{eqnarray}
\sigma_{xz}(ma)&=&m_{<0>}\ot_{B_x} m_{<1>}a;\eqlabel{2.1.1}\\
\sigma_{xx}(m_{<0>})\ot_{B_x} m_{<1>}&=&m_{<0>}\ot_{B_x} 1_x\ot_{B_x} m_{<1>};\eqlabel{2.1.2}\\
m_{<0>}m_{<1>}&=&m.\eqlabel{2.1.3}
\end{eqnarray}
A morphism between two descent data $(M,\sigma)$ and $(M',\sigma')$ is a morphism $f:\ M\to M'$
in $\Mm_k(X)_A$ such that
\begin{equation}\eqlabel{2.1.4}
f_{xx}(m_{<0>})\ot_{B_x}m_{<1>}=\sigma'_{xy}(f_{xy}(m)),
\end{equation}
for all $m\in M_{xy}$. The category of descent data is denoted $\Desc_B(A)$.
\end{definition}

\begin{proposition}\prlabel{2.2}
Let $i:\ B\to A$ be a $k$-linear $X$-functor.
We have an adjoint pair of functors $(F,G)$ between the categories $\Dd_k(X)_B$
and $\Desc_B(A)$.
\end{proposition}

\begin{proof}
For $N\in \Dd_k(X)_B$, we define $F(N)\in \Desc_B(A)$ as follows: $F(N)_{xy}=N_x\ot_{B_x} A_{xy}$,
and $\sigma_{xy}:\ F(N)_{xy}=N_x\ot_{B_x} A_{xy}\to F(N)_{xx}\ot_{B_x} A_{xy}=N_x\ot_{B_x} A_{xx}\ot_{B_x} A_{xy}$
is given by the formula
$$\sigma_{xy}(n\ot_{B_x} a)=n\ot_{B_x} 1_x\ot_{B_x} a.$$
Conditions \equref{2.1.1} and \equref{2.1.3} are obviously satisfied. We also compute easily that
\begin{eqnarray*}
&&\hspace*{-2cm}
\sigma_{xx}(n\ot_{B_x} 1_x)\ot_{B_x} a=n\ot_{B_x} 1_x\ot_{B_x} 1_x\ot_{B_x} a\\
&=& (n\ot_{B_x} a)_{<0>}\ot_{B_x} 1_x\ot_{B_x}(n\ot_{B_x} a)_{<1>},
\end{eqnarray*}
hence \equref{2.1.2} is also satisfied, and $F(N)$ is a descent datum. In particular, $F(B)=A$
is a descent datum, with $\sigma_{xy}:\ A_{xy}\to A_{xx}\ot_{B_x} A_{xy}$, $\sigma_{xy}(a)=1_x\ot_{B_x}a$.\\
At the level of morphisms, $F$ is defined as follows. Take $g:\ N\to N'$ in $\Dd_k(X)_B$. Then $F(g)=f:\
F(N)\to F(N')$ has $x,y$-component
$$f_{xy}=g_{xy}\ot_{B_x} A_{xy}:\ N_x\ot_{B_x} A_{xy}\to N'_x\ot_{B_x} A_{xy}.$$
Conversely, for $M\in \Desc_B(A)$, let $G(M)\in \Dd_k(X)$ be given by the formula
$$G(M)_x=\{m\in M_{xx}~|~\sigma_{xx}(m)=m\ot_{B_x} 1_x\}.$$
We claim that $G(M)\in \Dd_k(X)_B$, that is, $G(M)_x$ is a right $B_x$-module, for every $x\in X$.
Indeed, for every $m\in G(M)_x$ and $b\in B_x$, we have that $mb\in G(M)_x$ since
$\sigma_{xx}(mb)\equal{\equref{2.1.1}}m\ot_{B_x}1_x b=mb\ot_{B_x} 1_x$.
Observe that
$$G(A)_x=\{a\in A_{xx}~|~1_x\ot_{B_x}a=a\ot_{B_x}1_x\}.$$
$G(A)$ is a diagonal $k$-linear category: it is easy to show that every $G(A)_x$ is a $k$-algebra.
Also the algebra morphisms $i_x:\ B_x\to A_{xx}$ corestrict to $i_x:\ B_x\to G(A)_x$. Indeed,
for $b\in B_x$, we have that $i_x(b)=b1_x=1_xb$, and
$1_x\ot_{B_x} b1_x=b1_x\ot_{B_x} 1_x$, so that $i_x(b)\in G(A)_x$.\\
Let $f:\ (M,\sigma)\to (M',\sigma')$ be a morphism of descent data. For $m\in G(M)_x$, we have that
$$\sigma'_{xx}(f_{xx}(m))\equal{\equref{2.1.4}} f_{xx}(m_{<0>})\ot_{B_x}m_{<1>}=f_{xx}(m)\ot_{B_x} 1_x,$$
and $f_{xx}(m)\in G(M')_x$. We now define $G(f): G(M)\to G(M')$. $G(f)_x$ is the restriction 
and corestriction of $f_{xx}$ to $G(M)$ and $G(M')$.\\
\ul{Unit of the adjunction}. Let $N$ be a diagonal $B$-module. Then
\begin{eqnarray*}
&&\hspace*{-20mm}
GF(N)_x=
\{\sum_i n_i\ot_{B_x}a_i\in N_x\ot_{B_x} A_{xx}~|\\
&& 
\sum_i n_i\ot_{B_x} 1_x \ot_{B_x}a_i = \sum_i n_i\ot_{B_x}a_i\ot_{B_x} 1_x\}.
\end{eqnarray*}
$\eta^N:\ N\to GF(N)$ is now defined as follows:
$$\eta^N_x:\ N_x\to GF(N)_x\subset N_x\ot_{B_x} A_{xx},~~\eta^N_x(n)=n\ot_{B_x}1_x.$$
\ul{Counit of the adjunction}. Let $M$ be a descent datum. $FG(M)_{xy}=G(M)_x\ot_{B_x} A_{xy}$,
and $\varepsilon^M:\ FG(M)\to M$ is defined as follows:
$$\varepsilon^M_{xy}:\ G(M)_x\ot_{B_x} A_{xy}\to M_{xy},~~
\varepsilon^M_{xy}(m\ot_{B_x}a)=ma.$$
Verification of all the further details is left to the reader.
\end{proof}

\begin{proposition}\prlabel{2.3}
Let $i:\ B\to A$ be a $k$-linear $X$-functor. Take $x,y\in X$ and assume that $A_{xy}$ is flat as a left $B_x$-module.
Then the counit morphism $\varepsilon^M_{xy}$ from \prref{2.2} is bijective, for every descent datum $(M,\sigma)$.
\end{proposition}

\begin{proof}
Consider the map
$$i^M_x=M_{xx}\ot_{B_x}i_x:\ M_{xx}\to M_{xx}\ot_{B_x} A_{xx},~~
i^M_x(m)=m\ot_{B_x}1_x.$$
Then we have an exact sequence
$$0~\to~ G(M)_x~\mapright{\subset} ~M_{xx}
~\doublerightbis{\sigma_{xx}}{i^M_x}~
M_{xx}\ot_{B_x} A_{xx}.$$
By assumption, $A_{xy}$ is flat as a left $B_x$-module, hence the sequence
$$0~\to~ G(M)_x\ot_{B_x} A_{xy}~\mapright{\subset} ~M_{xx}\ot_{B_x} A_{xy}
~\doublerightbis{\sigma_{xx}\ot_{B_x} A_{xy}}{i^M_x\ot_{B_x} A_{xy}}~
M_{xx}\ot_{B_x} A_{xx}\ot_{B_x} A_{xy}$$
is exact. Take $m\in M_{xy}$. Then $\sigma_{xy}(m)=m_{<0>}\ot_{B_x} m_{<1>}\in M_{xx}\ot_{B_x} A_{xy}$
and
\begin{eqnarray*}
&&\hspace*{-2cm}
(\sigma_{xx}\ot_{B_x} A_{xy})(\sigma_{xy}(m))=
\sigma_{xx}(m_{<0>})\ot_{B_x} m_{<1>}\\
&\equal{\equref{2.1.2}}&m_{<0>}\ot_{B_x} 1_x \ot_{B_x} m_{<1>}=
(i^M_x\ot_{B_x} A_{xy})(\sigma_{xy}(m)).
\end{eqnarray*}
It follows that $\sigma_{xy}(m)\in G(M)_x\ot_{B_x} A_{xy}$, so $\sigma_{xy}$ corestricts to
$$\sigma_{xy}:\ M_{xy}\to G(M)_x\ot_{B_x} A_{xy}.$$
We now show that this map is the inverse of $\varepsilon^M_{xy}$. For all $m\in M_{xy}$, we have that
$$(\varepsilon^M_{xy}\circ \sigma_{xy})(m)=m_{<0>}m_{<1>}\equal{\equref{2.1.3}}m.$$
For $m\in G(M)_x$ and $a\in A_{xy}$, we easily calculate that
$$(\sigma_{xy}\circ \varepsilon^M_{xy})(m\ot_{B_x} a)=\sigma_{xy}(ma)\equal{\equref{2.1.1}}
m_{<0>}\ot_{B_x} m_{<1>}a=m\ot_{B_x} a.$$
\end{proof}

\begin{proposition}\prlabel{2.4}
Let $i:\ B\to A$ be a $k$-linear $X$-functor. Take $x\in X$ and assume that $A_{xx}$ is faithfully flat as a left $B_x$-module. Then $\eta^N_x$ is bijective for every $N \in \Dd_{k}(X)_{B}$.
\end{proposition}

\begin{proof}
Let $f_1,~f_2:\ N_x\ot_{B_x}A_{xx}\to N_x\ot_{B_x} A_{xx}\ot_{B_x} A_{xx}$ be defined by the formulas
$$f_1(n\ot_{B_x}a)=n\ot_{B_x}1_x\ot_{B_x}a~~;~~f_2(n\ot_{B_x}a)=n\ot_{B_x}a\ot_{B_x}1_x.$$
Since $\eta^N_x$ is the corestriction of $i^N_x$ to $GF(N)_x$, we have the following commutative diagram
$$\xymatrix{
0\ar[r]&N_x\ar[r]^(.4){i_x^N}\ar[d]^{\eta_x^N}&N_x\ot_{B_x}A_{xx}
\ar@<.5ex>[r]^(.38){f_1}\ar@<-.5ex>[r]_(.38){f_2}\ar[d]^{=}&N_x\ot_{B_x} A_{xx}\ot_{B_x} A_{xx}\ar[d]^{=}\\
0\ar[r]&GF(N)_x\ar[r]^(.4){\subset}&N_x\ot_{B_x}A_{xx}
\ar@<.5ex>[r]^(.38){f_1}\ar@<-.5ex>[r]_(.38){f_2}&N_x\ot_{B_x} A_{xx}\ot_{B_x} A_{xx}}$$
It follows from the definition of $G$ that the bottom row is exact. If we can show that the top row is exact, then
it will follow from the five lemma that  $\eta^N_x$ is an isomorphism. Since $A_{xx}$ is faithfully flat as a left
$B_x$-module, it suffices to show that the top row becomes exact after the functor $-\ot_{B_x}A_{xx,}$ is
applied to it.\\
Take $\alpha=\sum_i n_i\ot_{B_x} a_i\ot_{B_x} b_i\in N_x\ot_{B_x} A_{xx}\ot_{B_x} A_{xx}$, and assume
that
$(f_1\ot_{B_x} A_{xx})(\alpha)=(f_2\ot_{B_x} A_{xx})(\alpha)$,
that is,
$$\sum_i n_i\ot_{B_x} 1_x \ot_{B_x}a_i\ot_{B_x} b_i=\sum_i n_i\ot_{B_x} a_i\ot_{B_x}1_x\ot_{B_x} b_i.$$
Multiplying the third and the fourth tensor factor, we obtain that
$$(1_x^N\ot_{B_x}A_{xx})(\sum_i n_i\ot_{B_x} a_ib_i)=\sum_i n_i\ot_{B_x} 1_x \ot_{B_x}a_i b_i=\sum_i n_i\ot_{B_x} a_i\ot_{B_x} b_i=
\alpha,$$
so $\alpha\in \im(1_x^N\ot_{B_x}A_{xx})$, which is precisely what we need.
\end{proof}

As an immediate application of Propositions \ref{pr:2.2}-\ref{pr:2.4}, we obtain the following result, which can be viewed
as the faithfully flat descent theorem for $k$-linear categories. We would like to point out that \thref{2.5} can also be
derived from Beck's Theorem, see \cite[Sec. VI.7]{mclane}; we have prefered to present a direct proof.

\begin{theorem}\thlabel{2.5}
Let $i:\ B\to A$ be a $k$-linear $X$-functor. Assume that $A$ is locally flat as a left $B$-module. Then the following assertions are equivalent.
\begin{enumerate}
\item $A_{x,x}$ is faithfully flat as a left $B_{x}$-module for all $x \in X$;
\item the adjoint pair $(F,G)$
from \prref{2.2} is a pair of inverse equivalences, and the categories $\Dd_k(X)_B$
and $\Desc_B(A)$ are equivalent.
\end{enumerate}
\end{theorem}

\begin{proof}
We only need to prove $(2) \Rightarrow (1)$. Assume that $(F,G)$ is a pair of inverse equivalences. Fix $x \in X$ an let
\begin{equation}
0 \rightarrow N'_{x} \rightarrow N_{x} \rightarrow N''_{x} \rightarrow 0
\end{equation}
be a sequence of left $B_{x}$-modules such that
\begin{equation}
\eqlabel{2.5.1}
0 \rightarrow N'_{x} \otimes_{B_{x}} A_{xx} \rightarrow N_{x} \otimes_{B_{x}} A_{xx} \rightarrow N''_{x} \otimes_{B_{x}} A_{xx} \rightarrow 0
\end{equation}
is exact. If we apply the functor $G$ to the sequence, and use the fact that every $\eta^{N}_{x}$ is bijective, we find that \equref{2.5.1} is exact.
\end{proof}

For later use, we briefly discuss ${}_B\Desc(A)$, the category of left descent data.
Let $i:\ B\to A$ be as before. A left descent datum is a left $A$-module $M$
together with a family of linear maps $\tau=(\tau_{xy})_{x,y\in X}$,
$\tau_{xy}:\ M_{xy}\to A_{xy}\ot_{B_y} M_{yy}$, $\tau_{xy}(n)=n_{<-1>}\ot_{B_y}n_{<0>}$,
satisfying the following properties, for all $m\in M_{xy}$ and $a\in A_{zx}$:
\begin{eqnarray}
\tau_{xy}(am)&=&am_{<-1>}\ot_{B_y}m_{<0>};\eqlabel{2.6.1}\\
m_{<-1>}\ot_{B_y}\tau_{yy}(m_{<0>})&=& m_{<-1>}\ot_{B_y}1_y\ot_{B_y}m_{<0>}\eqlabel{2.6.2};\\
m_{<-1>}m_{<0>}&=&m\eqlabel{2.6.3}.
\end{eqnarray}
There is a pair of adjoint functors between ${}_B\Dd_k(X)$ and ${}_B\Desc(A)$,
which is a pair of inverse equivalences if every $A$ is locally faithfully flat as a right $B$-module.

\section{Relative Hopf modules}\selabel{3}
\subsection{Right relative Hopf modules}\selabel{3.1}
\begin{definition}\delabel{3.1}
Let $H$ be a $k$-linear semi-Hopf category. A right $H$-comodule category is a $k$-linear category $A$
with the following additional structure: every $A_{xy}$ is a right $H_{xy}$-comodule, with coaction
$\rho_{xy}:\ A_{xy}\to A_{xy}\ot H_{xy}$, that is compatible with the multiplication and unit on $A$
in the sense that
\begin{equation}\eqlabel{3.1.1}
\rho_{xz}(ab)=a_{[0]}b_{[0]}\ot a_{[1]}b_{[1]}~~{\rm and}~~\rho_{xx}(1_x^A)=1_x^A\ot 1_x^H,
\end{equation}
for all $a\in A_{xy}$ and $b\in A_{yz}$. We used the obvious Sweedler notation for the coaction.
A right relative $(A,H)$-Hopf module $M$ is a right $A$-module $M\in \Mm_k(X)_A$ such that every $M_{xy}$ is a right
$H_{xy}$-comodule satisfying the equation
\begin{equation}\eqlabel{3.1.2}
\rho_{xz}(ma)=m_{[0]}a_{[0]}\ot m_{[1]}a_{[1]},
\end{equation}
for all $m\in M_{xy}$ and $a\in A_{yz}$.\\
A morphism $f:\ M\to M'$ between two relative Hopf modules is a morphism $f:\ M\to M'$ that is
a morphism in $\Mm(X)_A$ and in in $\Mm(X)^H$. The category of relative Hopf modules is denoted $\Mm(X)^H_A$.
\end{definition}

Let $M$ be a relative $(A,H)$-Hopf module. For each $x\in X$, we consider $M^{{\rm co}H_{xx}}_{xx}$. Then
the $N_x$ are the components of an object $N=M^{{\rm coH}}\in \Dd_k(X)$. We then have that
$N_x=M^{{\rm coH}}_x=M^{{\rm co}H_{xx}}_{xx}$.
In particular $A$ is a relative Hopf module, and we can consider $A^{{\rm co}H}$. It is easy to see that every
$A^{{\rm co}H}_x$ is a $k$-algebra, so that $A^{{\rm co}H}$ is an algebra in $\Dd_k(X)$, and $M\in \Dd_k(X)_{A^{{\rm co}H}}$, for all $M\in \Mm(X)^{H}_{A}$.
Assume that $B$ is an algebra in $\Dd_k(X)$, and that we have an algebra morphism $i:\ B\to A^{{\rm co}H}$.
It follows that every relative Hopf module $M$ is a right $B$-module, by restriction of scalars.

\begin{proposition}\prlabel{3.2}
For any $H$-comodule algebra $A$,
we have a pair of adjoint functors $(F,G)$ between the categories $\Dd_k(X)_B$ and $\Mm(X)^H_A$.
\end{proposition}

\begin{proof}
We first define $F:\ \Dd_k(X)_B\to \Mm_k(X)^H_A$. For $N\in \Dd_k(X)_B$, we have $M=F(N)\in \Mm_k(X)^H_A$,
with $M_{xy}= N_x\ot_{B_x} A_{xy}$, with structure maps
$$(n\ot a)b=n\ot ab~~~;~~~\rho_{xy}(n\ot a)=n\ot a_{[0]}\ot a_{[1]},$$
for all $n\in N_x$, $a\in A_{xy}$ and $b\in A_{yz}$. For $g:\ N\to N'$ in $\Dd_k(X)_B$, $F(g)_{xy}= g_{x}\ot A_{xy}$.\\
For $M\in \Mm_k(X)^H_A$, let $G(M)=M^{{\rm co}H}$. Consider $f:\ M\to M'$ in $\Mm(X)^H_A$. It is easy to see
that $f_{xx}:\ M_{xx}\to M'_{xx}$ restricts and corestricts to a map $M_{xx}^{{\rm co}H_{xx}}\to {M'}_{xx}^{{\rm co}H_{xx}}$
which is by definition the $x$-component of $G(f)=f^{{\rm co}H}$.\\
Now we describe the unit and counit of the adjunction. For $N\in \Dd_k(X)_B$, $\eta^N:\ N\to GF(N)$ has components
$$\eta^N_x:\ N_x\to (N_x\ot _{A^{{\rm co}H}_x} A_{xx})^{{\rm co}H_{xx}},~~\eta^N_x(n)=n\ot 1_x.$$
For $M\in \Mm(X)^H_A$, $\varepsilon^M:\ FG(M)\to M$ has components
$$\varepsilon^M_{xy}:\ M^{{\rm co}H}_x\ot_{B_x} A_{xy}\to M_{xy},~~~\varepsilon^M_{xy}(m\ot_{B_x}a)=ma.$$
The verification of all further details is left to the reader.
\end{proof}

We will investigate when $(F,G)$ is a pair of inverse equivalences. We begin with some necessary conditions.
For $x,y,z\in X$, we consider the maps
$$\can^z_{xy}:\ A_{zx}\ot_{B_x} A_{xy}\to A_{zy}\ot H_{xy},~~~
\can^z_{xy}(a\ot a')=aa'_{[0]}\ot a'_{[1]}.$$

\begin{proposition}\prlabel{3.3}
We consider the pair of adjoint functors $(F,G)$ from \prref{3.2}.
\begin{enumerate}
\item If $F$ is fully faithful, then $i:\ B\to A^{{\rm co}H}$ is an isomorphism.
\item If $G$ is fully faithful, then all the $\can^z_{xy}$ are isomorphisms.
\end{enumerate}
\end{proposition}

\begin{proof}
1) If $F$ is fully faithful, then $\eta^N_x$ is an isomorphism, for all $x\in X$ and $N\in \Dd_k(X)_B$.
Take $N=B$. Then we have that
$$i_x=\eta^B_x:\ B_x\to (B_x\ot_{B_x} A_{xx})^{{\rm co}H_{xx}}=A_{xx}^{{\rm co}H_{xx}}=A^{{\rm co}H}_x$$
is an isomorphism.\\
2) Assume that $G$ is fully faithful. For each $z\in X$, consider $M^z\in \Mm_k(X)^H_A$ defined as follows:
$M^z_{xy}=A_{zy}\ot H_{xy}$, with structure maps
$$\rho_{xy}(a\ot h)=a\ot h_{(1)}\ot h_{(2)}~~~;~~~(a\ot h)a'=aa'_{[0]}\ot ha'_{[1]},$$
for all $a\in A_{zy}$, $h\in H_{xy}$, $a'\in A_{yu}$. We claim that
\begin{equation}\eqlabel{3.3.1}
(M^z_{xx})^{{\rm co}H_{xx}}\cong A_{zx}.
\end{equation}
It suffices to show that the maps
\begin{eqnarray*}
&& f:\ A_{zx}\to (M^z_{xx})^{{\rm co}H_{xx}},~~~f(a)=a\ot 1_x;\\
&&g: (M^z_{xx})^{{\rm co}H_{xx}}\to A_{zx},~~~g(\sum_i a_i\ot h_i)=\sum_i a_i\varepsilon_{xy}(h_i);
\end{eqnarray*}
are inverses. It is obvious that $g\circ f= A_{zx}$. Now take
$\sum_i a_i\ot h_i\in (M^z_{xx})^{{\rm co}H_{xx}}$. Then
$$\sum_i a_i\ot h_{i(1)}\ot h_{i(2)}=\sum_i a_i\ot h_i\ot 1_x.$$
Applying $\varepsilon_{xx}$ to the second tensor factor, we find that
$$\sum_i a_i\ot h_i=\sum_i a_i \varepsilon_{xx}(h_i)\ot 1_x=(f\circ g)(\sum_i a_i\ot h_i),$$
and this shows that $f\circ g= (M^z_{xx})^{{\rm co}H_{xx}}$. Finally observe that
\begin{eqnarray*}
&&\hspace*{-5mm}
\can^z_{xy}= \varepsilon^{M^z}_{xy}\circ (f\ot_{B_x}A_{xy})~:\\
&&
A_{zx}\ot_{B_x} A_{xy}\to (M^z_{xx})^{{\rm co}H_{xx}}\ot_{B_x} A_{xy}\to M^z_{xy}=A_{zy}\ot H_{xy}
\end{eqnarray*}
is an isomorphism. Indeed,
$$(\varepsilon^{M^z}_{xy}\circ (f\ot_{B_x}A_{xy}))(a\ot a')=(a\ot 1_x)a'=aa'_{[0]}\ot a'_{[1]}=\can^z_{xy}(a\ot a').$$
\end{proof}

\begin{proposition}\prlabel{3.4}
Let $H$ be a $k$-linear semi-Hopf category, let $A$ be a right $H$-comodule category, and let $B=A^{{\rm co}H}$.
Then we have a functor 
$$P:\ \Desc_B(A)\to \Mm_k(X)_A^H.$$
\end{proposition}

\begin{proof}
Take $(M,\sigma)\in \Desc_B(A)$, and consider
$$\rho_{xy}:\ M_{xy}\to M_{xy}\ot H_{xy},~~\rho_{xy}(m)=m_{<0>}m_{<1>[0]}\ot m_{<1>[1]}.$$
We will show that $(M,\rho)\in \Mm_k(X)_A^H$. Let us first show that $\rho_{xy}$ is coassociative.
\begin{eqnarray*}
&&\hspace*{-15mm}
((\rho_{xy} \ot A_{xy})\circ \rho_{xy})(m)=
\bigl(m_{<0>}m_{<1>[0]}\bigr)_{<0>}\bigl(m_{<0>}m_{<1>[0]}\bigr)_{<1>[0]}\\
&&~~~
\ot ~\bigl(m_{<0>}m_{<1>[0]}\bigr)_{<1>[1]} \ot m_{<1>[1]}\\
&\equal{\equref{2.1.1}}&
m_{<0><0>}\bigl(m_{<0><1>}m_{<1>[0]}\bigr)_{[0]}\ot \bigl(m_{<0><1>}m_{<1>[0]}\bigr)_{[1]}
\ot m_{<1>[1]}\\
&\equal{\equref{2.1.2}}&
m_{<0>} \bigl(1_xm_{<1>[0]}\bigr)_{[0]}\ot \bigl(1_xm_{<1>[0]}\bigr)_{[1]}\ot m_{<1>[1]}\\
&=& m_{<0>}m_{<1>[0]} \ot m_{<1>[1]}\ot m_{<1>[2]}
= ((M_{xy}\ot \Delta_{xy})\circ \rho_{xy})(m).
\end{eqnarray*}
We proceed with the counit property
$$((M_{xy}\ot \varepsilon_{xy})\circ \rho_{xy})(m)=
m_{<0>}m_{<1>[0]} \varepsilon_{xy}( m_{<1>[1]})=m_{<0>}m_{<1>}
\equal{\equref{2.1.2}}m.$$
Finally, the compatibility condition \equref{3.1.2} holds. For $m\in M_{xy}$ and $a\in A_{yz}$, we have that
\begin{eqnarray*}
\rho_{xz}(ma)
&\equal{\equref{2.1.1}}&
m_{<0>}\bigl(m_{<1>}a\bigr)_{[0]}\ot \bigl(m_{<1>}a\bigr)_{[1]}\\
&=& m_{<0>}m_{<1>[0]}a_{[0]}\ot m_{<1>[1]}a_{[1]}=m_{[0]}a_{[0]}\ot m_{[1]}a_{[1]}.
\end{eqnarray*}
We now define $P(M,\sigma)=(M,\rho)$. Let $f: (M,\sigma)\to (M',\sigma')$ be a morphism in $\Desc_B(A)$.
We claim that $f$ is also a morphism $(M,\rho)\to (M',\rho')$ in $\Mm_k(X)_A^H$. To this end, we need to show
that every $f_{xy}:\ M_{xy}\to M'_{xy}$ is $H_{xy}$-colinear. For all $m\in M_{xy}$, we have that
\begin{eqnarray*}
&&\hspace*{-15mm}
f_{xy}(m_{[0]})\ot m_{[1]}
= f_{xy}(m_{<0>}m_{<1>[0]}) \ot  m_{<1>[1]}\\
&\equal{*}& f_{xx}(m_{<0>})m_{<1>[0]} \ot  m_{<1>[1]}\\
&\equal{\equref{2.1.4}}&
f_{xy}(m)_{<0>} f_{xy}(m)_{<1>[0]}\ot f_{xy}(m)_{<1>[1]}
= \rho'_{xy}(f_{xy}(m)).
\end{eqnarray*}
At $*$, we used the fact that $f$ is right $A$-linear.
We now define $P(f)=f$.
\end{proof}

\begin{theorem}\thlabel{3.5}
Let $H$ be a $k$-linear semi-Hopf category, let $A$ be a right $H$-comodule category, and let $B=A^{{\rm co}H}$.
Then the following assertions are equivalent.
\begin{enumerate}
\item $\can^z_{xy}$ is bijective, for all $x,y,z\in X$;
\item $\can^y_{xy}$ is bijective, and $\can^x_{xy}$ has a left inverse $g_{xy}$, for all $x,y\in X$;
\item for all $x,y\in X$, there exists $\gamma_{xy}:\ H_{xy}\to A_{yx}\ot_{B_x} A_{xy}$, notation
$$\gamma_{xy}(h)=\sum_i l_i(h)\ot_{B_x} r_i(h),$$
such that
\begin{eqnarray}
\sum_i l_i(h)r_i(h)_{[0]} \ot r_i(h)_{[1]}&=& 1_y\ot h;\eqlabel{3.5.1}\\
\sum_ia_{[0]}l_i(a_{[1]}) \ot_{B_x} r_i(a_{[1]})&=& 1_x\ot_{B_x} a,\eqlabel{3.5.2}
\end{eqnarray}
for all $h\in H_{xy}$ and $a\in A_{xy}$.
\end{enumerate}
If these equivalent conditions are satisfied, then we call $A$ an $H$-Galois category extension of $B=A^{{\rm co}H}$.
\end{theorem}

\begin{proof}
$\ul{(1)\Rightarrow (2)}$ is trivial.\\
$\ul{(2)\Rightarrow (3)}$. We define $\gamma_{xy}$ by the formula
$$\gamma_{xy}(h)=(\can^y_{xy})^{-1}(1_y\ot h).$$
Then
$$1_y\ot h=\can^y_{xy}\bigl( \sum_i l_i(h)\ot_{B_x} r_i(h)\bigr)= \sum_i l_i(h)r_i(h)_{[0]} \ot r_i(h)_{[1]},$$
so that \equref{3.5.1} holds. Now we define
$f^z_{xy}:\ A_{zy}\ot H_{xy}\to A_{zx}\ot_{B_x} A_{xy}$ by the formula
\begin{equation}\eqlabel{3.5.3}
f^z_{xy}(a\ot h)= \sum_i al_i(h)\ot_{B_x} r_i(h).
\end{equation}
Then
$$(\can^z_{xy}\circ f^z_{xy})(a\ot h)=\sum_i al_i(h)r_i(h)_{[0]}\ot r_i(h)_{[1]}
\equal{\equref{3.5.1}} a1_y\ot h= a\ot h.$$
It follows that $f^z_{xy}$ is a right inverse of $\can^z_{xy}$. By assumption, $g_{xy}$ is a left inverse
of $\can^x_{xy}$, and it follows easily that
$$g_{xy}=g_{xy}\circ \can^x_{xy}\circ f^x_{xy}=f^x_{xy}.$$
For all $a\in A_{xy}$, we now have that
\begin{eqnarray*}
1_x\ot_{B_x} a&=& (f^x_{xy}\circ \can^x_{xy})(1_x\ot_{B_x} a)\\
&=& f^x_{xy}(a_{[0]}\ot a_{[1]})= \sum_ia_{[0]}l_i(a_{[1]}) \ot_{B_x} r_i(a_{[1]}),
\end{eqnarray*}
and this proves that \equref{3.5.2} holds.\\
$\ul{(3)\Rightarrow (1)}$. We define $f^z_{xy}$ using \equref{3.5.3}. We have shown above that $f^z_{xy}$
is a right inverse of $\can^z_{xy}$. It is also a left inverse since
\begin{eqnarray*}
&&\hspace*{-15mm}
(f^z_{xy} \circ \can^z_{xy})(a\ot_{B_x} a')=f^z_{xy}(aa'_{[0]}\ot a'_{[1]})\\
&=& \sum_i aa'_{[0]}l_i(a'_{[1]})\ot_{B_x} r_i(a'_{[1]})\equal{\equref{3.5.2}}
a1_x\ot_{B_x}a'=a\ot_{B_x}a',
\end{eqnarray*}
for all $a\in A_{zx}$ and $a'\in A_{xy}$.
\end{proof}

\begin{example}\exlabel{3.6}
Let $H$ be a $k$-linear semi-Hopf category. $H$ is a right $H$-comodule category, and $H^{{\rm co}H}=J$,
with $J_x=k$, for all $x\in X$. It follows from \cite[Theorem 9.2]{BCV} that $H$ is an $H$-Galois category extension
of $J$ if and only if $H$ is a Hopf category.
\end{example}

\begin{proposition}\prlabel{3.7}
Assume that the equivalent conditions of \thref{3.5} are satisfied. The maps $\gamma_{xy}$
have the following properties, for all $h\in H_{xy}$ and $h'\in H_{yz}$:
\begin{eqnarray}
&&\gamma_{xz}(hh')= \sum_{i,j} l_i(h')l_i(h)\ot_{B_x} r_j(h)r_i(h');\eqlabel{3.7.1}\\
&&\gamma_{xy}(h)\in (A_{yx}\ot_{B_x} A_{xy})^{B_y};\eqlabel{3.7.2}\\
&&\gamma_{xy}(h_{(1)})\ot h_{(2)}= \sum_i l_i(h)\ot_{B_x} r_i(h)_{[0]}\ot r_i(h)_{[1]};\eqlabel{3.7.3}\\
&&\sum_il_i(h)r_i(h)=\varepsilon_{xy}(h)1_y;\eqlabel{3.7.4}\\
&&\gamma_{xy}(h_{(2)})\ot S_{xy}(h_{(1)})= \sum_i l_i(h)_{[0]}\ot_{B_x} r_i(h)  \ot l_i(h)_{[1]}.\eqlabel{3.7.5}
\end{eqnarray}
For \equref{3.7.5}, we need the additional assumption that $H$ is a Hopf category.
\end{proposition}

\begin{proof}
\begin{eqnarray*}
&&\hspace*{-2cm}
\can^z_{xy}\Bigl(\sum_{i,j} l_i(h')l_i(h)\ot_{B_x} r_j(h)r_i(h')\Bigr)\\
&=& \sum_{i,j} l_i(h')l_i(h)r_j(h)_{[0]}r_i(h')_{[0]}\ot r_j(h)_{[1]}r_i(h')_{[1]}\\
&\equal{\equref{3.5.1}}& \sum_i l_i(h')1_y r_i(h')_{[0]}\ot h r_i(h')_{[1]}
\equal{\equref{3.5.1}} 1_z\ot hh',
\end{eqnarray*}
proving \equref{3.7.1}. \equref{3.7.2} follows if we can show that
$$\sum_i l_i(h)\ot r_i(h)b=\sum_i bl_i(h)\ot r_i(h),$$
for all $b\in B_y$. Indeed,
\begin{eqnarray*}
&&\hspace*{-2cm}
\can^y_{xy}\Bigl(\sum_i l_i(h)\ot r_i(h)b\Bigr)
= \sum_i l_i(h) r_i(h)_{[0]}b_{[0]}\ot _i(h)_{[1]}b_{[1]}\\
&\equal{\equref{3.5.1}}&1_yb\ot h=b1_y\ot h \equal{\equref{3.5.1}} \can^y_{xy}\Bigl(b\sum_i l_i(h)\ot r_i(h)\Bigr).
\end{eqnarray*}
 \equref{3.7.3} is proved in a similar way:
 \begin{eqnarray*}
 &&\hspace*{-2cm}
 (\can^y_{xy} \ot H_{xy}))\Bigl(\sum_i l_i(h)\ot_{B_x} r_i(h)_{[0]}\ot r_i(h)_{[1]}\Bigr)\\
 &=& \sum_i l_i(h) r_i(h)_{[0]}\ot r_i(h)_{[1]}\ot r_i(h)_{[2]}\\
 &=& \sum_i l_i(h) r_i(h)_{[0]}\ot \Delta_{xy}(r_i(h)_{[1]})\\
&\equal{\equref{3.5.1}}& 1_y\ot \Delta_{xy}(h)=(\can^y_{xy} \ot H_{xy}))\Bigl( \gamma_{xy}(h_{(1)})\ot h_{(2)}\Bigr).
\end{eqnarray*}
 \equref{3.7.4} follows after we apply $A_{yy}\ot \varepsilon_{xy}$ to \equref{3.5.1}. \equref{3.7.5}
 is equivalent to
 \begin{equation}\eqlabel{3.7.6}
 1\ot h_{(2)} \ot S_{xy}(h_{(1)})= \sum_i l_i(h)_{[0]} r_i(h)_{[0]}  \ot r_i(h)_{[1]}  \ot l_i(h)_{[1]}.
 \end{equation}
 Indeed, applying $\can^y_{xy}\ot H_{yx}$ to \equref{3.7.5}, we obtain \equref{3.7.6}. Applying
 $\rho_{yy}\ot ((S_{xy}\ot H_{xy})\circ \Delta_{xy})$ to \equref{3.5.1}, we obtain that
 \begin{eqnarray*}
 &&\hspace*{-1cm}
 \sum_i l_i(h)_{[0]} r_i(h)_{[0]}  \ot l_i(h)_{[1]} r_i(h)_{[1]}\ot S_{xy}(r_i(h)_{[2]})\ot r_i(h)_{[3]}\\
 &=& 1_y^A\ot 1_y^H\ot S_{xy}(h_{(1)})\ot h_{(2)}.
 \end{eqnarray*}
 Now we multiply the second and third tensor factor, and obtain that
 $$\sum_i l_i(h)_{[0]} r_i(h)_{[0]}  \ot l_i(h)_{[1]} \ot r_i(h)_{[1]}=1_y^A\ot S_{xy}(h_{(1)})\ot h_{(2)}.$$
 \equref{3.7.6} follows after we switch the second and the third tensor factor.
 \end{proof}

\leref{3.8} is folklore; we will need it in the proof of \thref{3.9}, and this is why we provide a detailed proof. 

\begin{lemma}\lelabel{3.8}
Let $A$ be a $k$-algebra, and take $M\in \Mm_A$, $N\in {}_A\Mm_A$, $P\in {}_A\Mm$. For $m_i\in M$,
$n_i\in N^A$, $p_i\in P$, we have the following implication:
$$\sum_i n_i\ot m_i\ot_A p_i=0~{\rm in}~N^A\ot M\ot_A P~~\Longrightarrow~
\sum_i m_i\ot_A n_i\ot_A p_i=0~{\rm in}~M\ot_AN\ot_A P.$$
\end{lemma}

\begin{proof}
From the assumption that $\sum_i n_i\ot m_i\ot_A p_i=0$, it follows that
$$\sum_i n_i\ot m_i\ot p_i=\sum_j x_j \ot y_ja_j\ot z_j-x_j \ot y_j\ot a_jz_j,$$
for some $x_j\in N^A$, $y_j\in M$, $a_j\in A$ and $z_j\in P$. This implies that
\begin{eqnarray*}
\sum_i m_i\ot n_i\ot p_i&=&
\sum_j y_ja_j\ot x_j \ot z_j-y_j\ot a_jx_j \ot z_j\\
&+& y_j\ot x_ja_j \ot z_j-y_j\ot x_j \ot a_jz_j,
\end{eqnarray*}
where we used the fact that $x_j\in N^A$. This implies that $\sum_i m_i\ot_A n_i\ot_A p_i=0~{\rm in}~M\ot_AN\ot_A P.$
\end{proof}

\begin{theorem}\thlabel{3.9}
Let $H$ be a Hopf category, and let $A$ be 
an $H$-Galois category extension of $B=A^{{\rm co}H}$.
Then the functor $P:\ \Desc_B(A)\to \Mm_k(X)_A^H$ from \prref{3.4} is an isomorphism of categories.
\end{theorem}

\begin{proof}
We define a functor $Q:\ \Mm_k(X)_A^H\to \Desc_B(A)$. For a relative Hopf module $(M,\rho)$, consider
$$\sigma_{xy}=(\psi_{x,y,x}\ot_{B_x}A_{xy})\circ (M_{xy}\ot \gamma_{xy})\circ \rho_{xy}:\
M_{xy}\to M_{xx}\ot_{B_x}A_{xy},$$
that is,
$$\sigma_{xy}(m)=m_{<0>}\ot_{B_x}m_{<1>}=m_{[0]}\gamma_{xy}(m_{[1]})=\sum_i m_{[0]}l_i(m_{[1]})\ot_{B_x}r_i(m_{[1]}).$$
We claim that $(M,\sigma)\in \Desc_B(A)$. We will first show that \equref{2.1.1} holds, that is
$$\sigma_{xz}(ma)=m_{<0>}\ot_{B_x} m_{<1>}a,$$
for all $m\in M_{xy}$ and $a\in A_{yz}$. It follows from \equref{3.5.2} that
$$\sum_ia_{[0]}l_i(a_{[1]}) \ot_{B_y} r_i(a_{[1]})= 1_y\ot_{B_y} a.$$
From \equref{3.7.2}, we know that $\gamma_{xy}(h)\in (A_{yx}\ot_{B_x}A_{xy})^{B_y}$, for all $h\in H_{xy}$.
Therefore we have that
$$\gamma_{xy}(h)\ot \sum_ia_{[0]}l_i(a_{[1]}) \ot_{B_y} r_i(a_{[1]})=
\gamma_{xy}(h)\ot 1_y\ot_{B_y} a.$$
in $(A_{yx}\ot_{B_x}A_{xy})^{B_y}\ot A_{xy}\ot_{B_y} A_{xy}$. From \leref{3.8}, it follows that
$$\sum_ia_{[0]}l_i(a_{[1]})\ot_{B_y} \gamma_{xy}(h)
 \ot_{B_y} r_i(a_{[1]})= 1_y\ot_{B_y} \gamma_{xy}(h)\ot_{B_y} a$$
in $A_{xy}\ot_{B_y} A_{yx}\ot_{B_x}A_{xy} \ot_{B_y} A_{yx}$. Multiplying the two first and the two last tensor
factors, we find that
\begin{equation}\eqlabel{3.9.1}
\sum_{i,j} a_{[0]}l_i(a_{[1]})l_j(h)\ot_{B_x} r_j(h)r_i(a_{[1]})= \sum_j l_j(h)\ot_{B_x} r_j(h)a.
\end{equation}
Finally
\begin{eqnarray*}
\sigma_{xz}(ma)&=&
\sum_i m_{[0]}a_{[0]}l_i(m_{[1]}a_{[1]})\ot_{B_x}r_i(m_{[1]}a_{[1]})\\
&\equal{\equref{3.7.1}}& 
\sum_{i,j} m_{[0]}a_{[0]}l_i(a_{[1]})l_j(m_{[1]}) \ot_{B_x} r_j(m_{[1]})r_i(a_{[1]})\\
&\equal{\equref{3.9.1}}&
\sum_j m_{[0]}1_y l_j(m_{[1]}) \ot_{B_x} r_j(m_{[1]})a=\sigma_{xy}(m)a,
\end{eqnarray*}
and \equref{2.1.1} follows. Our next aim is to show that \equref{2.1.2} holds. Take $m\in M_{xy}$. It
follows from \equref{3.7.5} that
\begin{eqnarray*}
&&\hspace*{-2cm}
m_{[0]}\ot m_{[1]}\ot \gamma_{xy}(m_{[3]}) \ot S_{xy}(m_{[2]})\\
&=&m_{[0]}\ot m_{[1]}\ot l_i(m_{[2]})_{[0]}\ot_{B_x} r_i(m_{[2]})\ot l_i(m_{[2]})_{[1]},
\end{eqnarray*}
and
\begin{eqnarray*}
&&\hspace*{-15mm}
m_{[0]}\ot m_{[1]}S_{xy}(m_{[2]})\ot \gamma_{xy}(m_{[3]})
= \sum_j m_{[0]}\ot 1_x^H\ot l_j(m_{[1]})\ot_{B_x} r_j(m_{[1]})\\
&=& m_{[0]}\ot m_{[1]}l_i(m_{[2]})_{[1]} \ot l_i(m_{[2]})_{[0]}\ot_{B_x} r_i(m_{[2]}).
\end{eqnarray*}
Now we apply $\gamma_{xx}:\ H_{xx}\to (A_{xx}\ot_{B_x} A_{xx})^{B_x}$ (see \equref{3.7.2})
to the second tensor factor. Observing that $\gamma_{xx}(1_x^H)=1_x^A\ot_{B_x}1_x^A$,
this gives us the following equality in
$A_{xy}\ot (A_{xx}\ot_{B_x} A_{xx})^{B_x} \ot_{B_y} A_{yx}\ot_{B_x} A_{xy}$:
\begin{eqnarray*}
&&\hspace*{-15mm}
\sum_j m_{[0]}\ot 1_x^A\ot_{B_x}1_x^A \ot l_j(m_{[1]})\ot_{B_x} r_j(m_{[1]})\\
&=& \sum_i m_{[0]}\ot \gamma_{xx}(m_{[1]}l_i(m_{[2]})_{[1]}) \ot l_i(m_{[2]})_{[0]}\ot_{B_x} r_i(m_{[2]}).
\end{eqnarray*}
Now we apply \leref{3.8}, and obtain the equality
\begin{eqnarray*}
&&\hspace*{-15mm}
\sum_j m_{[0]}\ot l_j(m_{[1]})\ot_{B_x} 1_x^A\ot_{B_x}1_x^A \ot_{B_x} r_j(m_{[1]})\\
&=&\sum_i m_{[0]}\ot l_i(m_{[2]})_{[0]}\ot_{B_x} \gamma_{xx}(m_{[1]}l_i(m_{[2]})_{[1]})\ot_{B_x} r_i(m_{[2]})
\end{eqnarray*}
in $A_{xy}\ot A_{yx}\ot_{B_x} A_{xx}\ot_{B_x} A_{xx}\ot_{B_x} A_{xy}$. Multiplying the first three tensor
factors, we obtain that
\begin{eqnarray*}
&&\hspace*{-15mm}
\sum_j m_{[0]} l_j(m_{[1]})\ot_{B_x}1_x^A \ot_{B_x} r_j(m_{[1]})\\
&=&\sum_i m_{[0]} l_i(m_{[2]})_{[0]} \gamma_{xx}(m_{[1]}l_i(m_{[2]})_{[1]})\ot_{B_x} r_i(m_{[2]}),
\end{eqnarray*}
which is precisely \equref{2.1.2}. \equref{2.1.3} follows easily:
$$m_{<0>}m_{<1>}=\sum_i m_{[0]}l_i(m_{[1]})r_i(m_{[1]})\equal{\equref{3.7.4}} m_{[0]}\varepsilon_{xy}(m_{[1]})1_y=m.$$
We now define $Q(M,\rho)=(M,\sigma)$. If $f:\ (M,\rho)\to (M',\rho')$ is a morphism in $\Mm_k(X)_A^H$, then
it is also a morphism $(M,\sigma)\to (M',\sigma')$ in $\Desc_B(A)$. Indeed, for all $m\in M_{xy}$, we have that
\begin{eqnarray*}
&&\hspace*{-2cm}
f_{xx}(m_{<0>})\ot_{B_x}m_{<1>}
= \sum_i f_{xx}(m_{[0]}l_i(m_{[1]}))\ot_{B_x}r_i(m_{[1]})\\
&=& \sum_i f_{xx}(m_{[0]})l_i(m_{[1]})\ot_{B_x}r_i(m_{[1]})\\
&=& \sum_i (f_{xx}(m))_{[0]}l_i((f_{xx}(m))_{[1]})\ot_{B_x}r_i((f_{xx}(m))_{[0]})\\
&=&\sigma'_{xy}(f_{xy}(m)),
\end{eqnarray*}
so that \equref{2.1.4} holds. We now define $Q(f)=f$.\\
Take $(M,\sigma)\in \Desc_B(A)$, and let $(Q\circ P)(M,\sigma)=(M,\sigma')$. Then for all $m\in M_{xy}$,
we have that
\begin{eqnarray*}
\rho_{xy}(m)&=&m_{<0>}m_{<1>[0]}\ot m_{<1>[1]};\\
\sigma'_{xy}(m)&=& \sum_i m_{<0>}m_{<1>[0]}l_i(m_{<1>[1]}) \ot_{B_x} r_i(m_{<1>[1]})\\
&\equal{\equref{3.5.2}}& \sum_i m_{<0>}1_x \ot_{B_x} m_{<1>}=\sigma_{xy}(m)
\end{eqnarray*}
Finally take $(M,\rho)\in \Mm_k(X)_A^H$, and let $(P\circ A)(M,\rho)=(M,\rho')$. Then for all
$m\in M_{xy}$
\begin{eqnarray*}
\sigma_{xy}(m)&=&\sum_i m_{[0]}l_i(m_{[1]})\ot_{B_x}r_i(m_{[1]});\\
\rho'_{xy}(m)&=&\sum_i m_{[0]}l_i(m_{[1]})r_i(m_{[1]})_{[0]}\ot r_i(m_{[1]})_{[1]}\\
&\equal{\equref{3.5.1}}& m_{[0]}1_y \ot m_{[1]}= \rho_{xy}(m).
\end{eqnarray*}
This shows that $Q$ is the inverse of $P$.
\end{proof}

\subsection{Left relative Hopf modules}\selabel{3.2}
As in \seref{3.1}, let $H$ be a $k$-linear semi-Hopf category, and let $A$ be a right $H$-comodule category.
We introduce left $(A,H)$-relative Hopf modules. The results of \seref{3.1} have their counterparts for
left $(A,H)$-relative Hopf modules. The proofs are similar, so we restrict to a brief survey of the results.
 A left $(A,H)$-relative Hopf module is an object $M\in {}_A\Mm_k(X)$ such that every $M_{xy}$ is a right
 $H_{xy}$-comodule satisfying the compatibility relations
 $$\rho_{xz}(am)=a_{[0]}m_{[0]}\ot a_{[1]}m_{[1]},$$
 for all $a\in A_{xy}$ and $m\in M_{yz}$. The category of left $(A,H)$-relative Hopf modules is denoted
 as ${}_A\Mm_k(X)^H$. As in \seref{3.1}, we assume that $B$ is an algebra in $\Dd_k(X)$, and that $i:\ B\to A^{{\rm co}H}$
 is an algebra morphism.
 
 \begin{proposition}\prlabel{3.10}
 We have a pair of adjoint functors $(F',G')$ between ${}_B\Dd_k(X)$ and ${}_A\Mm_k(X)^H$.
 \end{proposition}
 
 \begin{proof}
 For $N\in {}_B\Dd_k(X)$, $F(N)_{xy}=A_{xy}\ot N_y$, with action and coaction given by the formulas
 $$a'(a\ot n)= a'a\ot n~~~{\rm and}~~~\rho_{xy}(a\ot n)=a_{[0]}\ot n\ot a_{[1]},$$
 for all $a'\in A_{ux}$, $a\in A_{xy}$ and $n\in N_y$. For a relative Hopf module $M$,
 $G'(M)=M^{{\rm co}H}$. The unit $\eta^N:\ N\to G'F'(N)$ and the counit $\varepsilon^M:\ F'G'(M)\to M$ are
 given by the formulas
$$ \begin{array}{lccl}
\eta^N_x:\ N_x\to (A_{xx}\ot_{B_x} N_x)^{{\rm co}H_{xx}}&;&\eta^N_x(n)=n\ot_{B_x}1_x;\\
\varepsilon^M_{xy}:\ A_{xy}\ot_{B_y} M^{{\rm co}H_{xx}}_{xx}&;& \varepsilon^M_{xy}(a\ot_{B_y}m)=am.
\end{array}$$
\end{proof}

For all $y\in X$, we consider the relative Hopf module $M^y$, $M^y_{zx}= A_{zy}\ot H_{zx}$, with action
and coaction given by the formulas
$$a'(a\ot h)=a'_{[0]}a\ot a'_{[1]}h~~~{\rm and}~~~\rho_{zx}(a\ot h)=a\ot h_{(1)}\ot h_{(2)},$$
for $a'\in A_{uz}$, $a\in A_{zy}$ and $h\in H_{zx}$. We have an isomorphism
$$f:\ A_{xy}\to \bigl(M^y_{xx}\bigr)^{{\rm co}H_{xx}},~~~f(a)=a\ot 1_x,$$
with inverse given by the formula $f^{-1}(\sum_i a_i\ot h_i)=\sum_i a_i\varepsilon_{xx}(h_i)$. Now observe that the composition
$$\can'^{y}_{zx}=\varepsilon^{M^y}_{zx}\circ (A_{zx}\ot_{B_x}f):\ A_{zx}\ot_{B_x}A_{xy}\to M^y_{zx}= A_{zy}\ot H_{zx},$$
is given by the formula
\begin{equation}\eqlabel{3.11.0}
\can'^{y}_{zx}(a\ot_{B_x}a')=a_{[0]}a'\ot a_{[1]}.
\end{equation}

 \begin{proposition}\prlabel{3.11}
 If $F'$ is fully faithful, then $i:\ B\to A^{{\rm co}H}$ is an isomorphism; If $F'$ is fully faithful, then $\can'^{y}_{zx}$
 is an isomorphism, for all $x,y,z\in X$.
 \end{proposition}

\begin{theorem}\thlabel{3.12}
Let $H$ be a $k$-linear semi-Hopf category, let $A$ be a right $H$-comodule category, and let $B=A^{{\rm co}H}$.
Then the following assertions are equivalent.
\begin{enumerate}
\item $\can'^{y}_{zx}$ is bijective, for all $x,y,z\in X$;
\item $\can'^{z}_{zx}$ is bijective and $\can'^{x}_{zx}$ has a left inverse $g'_{zx}$, for all $x,z\in X$;
\item for all $x,z\in X$, there exists $\gamma'_{zx}:\ H_{zx}\to A_{zx}\ot_{B_x} A_{xz}$, notation
$$\gamma'_{zx}(h)=\sum_i l'_i(h)\ot_{B_x} r'_i(h),$$
such that
\begin{eqnarray}
\sum_i l'_i(h)_{[0]}r'_i(h) \ot l'_i(h)_{[1]}&=& 1_z\ot h;\eqlabel{3.11.1}\\
\sum_i l'_i(a_{[1]}) \ot_{B_x} r'_i(a_{[1]}) a_{[0]}&=& a\ot_{B_x} 1_x,\eqlabel{3.11.2}
\end{eqnarray}
for all $h\in H_{zx}$ and $a\in A_{zx}$.
\end{enumerate}
If these equivalent conditions are satisfied, then we call $A$ an $H$-Galois' category extension of $B=A^{{\rm co}H}$.
\end{theorem}

\begin{proof}
$\ul{(2)\Rightarrow (3)}$. For $h\in H_{zx}$, we define
$$\gamma'_{zx}(h)=\bigl(\can'^{z}_{zx}\bigr)^{-1}(1_z\ot h).$$
$\ul{(2)\Rightarrow (3)}$. For $a\in A_{zy}$ and $h\in H_{zx}$, we define
$$\bigl(\can'^{y}_{zx}\bigr)^{-1}(a\ot h)=\sum_i l'_i(h)\ot_{B_x} r'_i(h)a.$$
\end{proof}

\begin{theorem}\thlabel{3.13}
Let $H$ be a $k$-linear semi-Hopf category, let $A$ be a right $H$-comodule category, and let $B=A^{{\rm co}H}$.
We have a functor $P':\ {}_B\Desc(A)\to {}_A\Mm_k(X)^H$. If $A$ an $H$-Galois' category extension of $B$,
then $P'$ is an isomorphism of categories.
\end{theorem}

\begin{proof}
For a descent datum $(M,\tau)$, we define $P'(M,\tau)=(M,\rho)$, with
$$\rho_{xy}(m)=m_{<-1>[0]}m_{<0>}\ot m_{<-1>[1]},$$
for $m\in M_{xy}$. Assume that  $A$ an $H$-Galois' category extension of $B$. For
$(M,\rho)\in {}_A\Mm_k(X)^H_A$, define $Q'(M,\rho)=(M,\tau)$, with
$$\tau_{zx}:\ M_{zx}\to A_{zx}\ot_{B_x} M_{xx},~~~\tau_{zx}(m)=\sum_i l'_i(m_{[1]})\ot_{B_x} r'_i(m_{[1]})m_{[0]}.$$
Then $Q'$ is the inverse of $P'$.
\end{proof}

\begin{theorem}\thlabel{3.14}
Let $H$ be a $k$-linear Hopf category with bijective antipode, let $A$ be a right $H$-comodule category, and let $B=A^{{\rm co}H}$.
Then $A$ is an $H$-Galois' category extension of $B$ if and only if $A$ an $H$-Galois category extension of $B$.
\end{theorem}

\begin{proof}
The map 
$$\phi:\ A_{zy}\ot H_{xy}\to A_{zy}\ot H_{zx},~~\phi(a\ot h)=a_{[0]}\ot a_{[1]}S_{xy}(h),$$
is invertible, with inverse given by the formula
$$\phi^{-1}(a\ot h')=a_{[0]}\ot S^{-1}_{xz}(h').$$
An easy computation shows that 
$$\can'^{y}_{xz}=\phi\circ \can^z_{xy}.$$
Consequently $\can'$ is invertible if and only if $\can$ is invertible.
\end{proof}

\section{$k$-linear Clusters}\selabel{4}
\begin{definition}\delabel{4.1}
A $k$-linear cluster $\Aa$ with underlying class $X$ consists of a class of $k$-linear categories with underlying
class $X$, indexed by $X$, that is, for every $x\in X$, we have a $k$-linear category $\Aa^x$.
\end{definition}

Let $\Aa$ be a cluster. A right $\Aa$-module is an object $M\in \Mm_k(X)$, together with morphisms
$$\psi_{xyz}:\ M_{xy}\ot \Aa^x_{yz}\to M_{xz},~~\psi_{xyz}(m\ot a)=m a,$$
satisfying the appropriate associativity and unit conditions:
$$m(ab)=(ma)b~~{\rm and}~~m1^x_y=m,$$
for all $m\in M_{xy}$, $a\in \Aa^{x}_{yz}$ and $b\in \Aa^{x}_{zu}$. $1^x_y$ is the unit element of  $\Aa^{x}_{yy}$.
A morphism $\varphi:\ M\to N$ between two right $\Aa$-modules $M$ and $N$ is a morphism 
$\varphi:\ M\to N$ in $\Mm_k(X)$ that is right $\Aa$-linear, which means that
$$\varphi_{xz}(ma)=\varphi_{xy}(m)a,$$
for all $m\in M_{xy}$ and $a\in \Aa^x_{yz}$. The category of right $\Aa$-modules will be denoted by $\Mm_\Aa$.

\begin{example}\exlabel{4.2}
Let $B$ be a diagonal $k$-linear category, and let $M\in \Mm_k(X)$ be a left $B$-module, meaning that we have
maps $B_x\ot M_{xy}\to M_{xy}$ satisfying the appropriate associativity and unit conditions. We have a cluster
$\Aa={}_B\End(M)$ defined as follows:
$$\Aa^x_{yz}={}_B\End(M)^x_{yz}={}_{B_x}\Hom(M_{xy},M_{xz})$$
The multiplication maps are given by opposite composition: for $f\in \Aa^x_{yz}$ and $g\in \Aa^x_{zu}$, we put
$fg=g\circ f$. The unit element of $\Aa^x_{yy}$ is the identity $M_{xy}$. ${}_B\End(M)$ is called the left endocluster
of $M$. Note that $M$ is a right ${}_B\End(M)$-module, via the structure maps
$$M_{xy}\ot \Aa^x_{yz}\to M_{xz},~~~m\ot f=f(m).$$
The right endocluster ${\Aa'}=\End_B(M)$ of a right $B$-module $M$ can be defined in a similar way:
$$\Aa'^x_{yz}=\End_B(M)^x_{yz}=\Hom_{B_x}(M_{zx},M_{yx}).$$
Now the multilplcation is given by composition.
\end{example}

More examples will be presented in Sections \ref{se:6} and \ref{se:7}.

\section{Faithfully projective descent}\selabel{5}
\begin{proposition}\prlabel{5.1}
We consider the setting of \seref{2}:
$A$ and $B$ are $k$-linear categories, with underlying class $X$, and $B$ is a diagonal algebra.
$i:\ B\to A$ is a $k$-linear $X$-functor. Then $A$ is a left $B$-module via restriction of scalars,
and we can consider $\Aa={}_B\End(A)$ as in \exref{4.2}.
We have a functor $H:\ \Desc_B(A)\to \Mm_\Aa$.
\end{proposition}

\begin{proof}
Let $(M,\sigma)\in \Desc_B(A)$. We define a right $\Aa$-action on $M$ as follows:
for $m\in M_{xy}$ and $f\in \Aa^x_{yz}$, let $mf=m_{<0>}f(m_{<1>})\in M_{xz}$.
Let us show that the associativity and unit condition are satisfied. Take $g\in \Aa^x_{zu}$.
\begin{eqnarray*}
(mf)g&=& (m_{<0>}f(m_{<1>}))_{<0>}g\bigl((m_{<0>}f(m_{<1>}))_{<1>}\bigr)\\
&\equal{\equref{2.1.1}}& m_{<0><0>}g\bigl(m_{<0><1>}f(m_{<1>})\bigr)\\
&\equal{\equref{2.1.2}}& m_{<0>}(g\circ f)(m_{<1>})=m(fg);\\
mA_{xy}&=&m_{<0>}m_{<1>}\equal{\equref{2.1.3}}m.
\end{eqnarray*}
Now we define $H(M,\sigma)=M$. $H$ acts as the identity on morphisms: if $\varphi:\ (M,\sigma)\to (M',\sigma')$ is a morphism
in $\Desc_B(A)$, then $\varphi$ is right $\Aa$-linear. Indeed, for $m\in M_{xy}$ and
$f\in \Aa^x_{yz}$, we have that
\begin{eqnarray*}
\varphi_{xy}(m)f&=& \varphi_{xy}(m)_{<0>}f(\varphi_{xy}(m)_{<1>})\\
&\equal{\equref{2.1.4}}&\varphi_{xx}(m_{<0>})f(m_{<1>})\\
&=& \varphi_{xx}(m_{<0>})f(m_{<1>}))=\varphi_{xz}(mf).
\end{eqnarray*}
\end{proof}

\begin{remark}\relabel{5.2}
To any $a\in A_{yz}$, we can associate $r_a\in \Aa^x_{yz}={}_{B_x}\Hom(A_{xy},A_{xz})$, given by
right mulitplication by $a$: $r_a(a')=a'a$, for all $a'\in A_{xy}$. For $a'\in A_{xy}$ we have that
$$a'r_a=a'_{<0>}r_a(a'_{<1>})= a'_{<0>}a'_{<1>}a=a'a.$$
\end{remark}

\begin{theorem}\thlabel{5.3}
Let $A$ and $B$ be as in \prref{5.1}, and assume that $A$ is locally finite as a left
$B$-module. Then the categories $\Desc_B(A)$ and $\Mm_\Aa$ are isomorphic.
\end{theorem}

\begin{proof}
We will show that the functor $H$ from \prref{5.1} has an inverse $K$. Take a right $\Aa$-module $M$.
Then $M$ is also a right $A$-module: for $m\in M_{xy}$ and  $a\in A_{yz}$, we just put
$ma=mr_a$, see \reref{5.2}.\\
Let $\sum_i e_i^*\ot_{B_x} e_i$ be a finite dual basis of the left $B_x$-module $A_{xy}$. For $a^*\in A_{xy}^*$,
we have that $i_x\circ a^*\in \Aa^x_{yx}$. Now we define $\sigma_{xy}:\ M_{xy}\to M_{xx}\ot_{B_{x}}A_{xy}$
as follows:
$$\sigma_{xy}(m)=\sum_i m(i_x\circ e_i^*) \ot_{B_x} e_i.$$
We claim that $(M,\sigma)\in \Desc_B(A)$. We need to show that (\ref{eq:2.1.1}-\ref{eq:2.1.3}) hold. Take
$m\in M_{xy}$ and $a\in A_{yz}$, and let $\sum_j f_j^*\ot_{B_x} f_j$ be a finite dual basis for $A_{xz}$.
\equref{2.1.1} follows if we can show that
$$\sigma_{xz}(ma)=\sum_j (mr_a)(i_x\circ f_j^*)\ot_{B_x}f_j$$
equals
$$\sigma_{xz}(m)a=\sum_i m(i_x\circ e_i^*) \ot_{B_x} e_ia$$
in $M_{xx}\ot_{B_x} A_{xz}\cong \Hom_{B_x}(A_{xz}^*,M_{xx})$, see \equref{1.2.2}. To this end, it suffices
to show that both terms are equal after we evaluate them at an arbitrary $a^*\in A_{xz}^*$, that is,
$$\sum_j ((mr_a)(i_x\circ f_j^*))a^*(f_j)=m\Bigl(\sum_j r_{(i_x\circ a^*)(f_j)}\circ (i_x\circ f_j^*)\circ r_a\Bigr)$$
equals
$$\sum_i (m(i_x\circ e_i^*)) a^*(e_ia)=m\Bigl(\sum_i r_{a^*(e_ia)}\circ i_x\circ e_i^*\Bigr).$$
It suffices to show that
$$\sum_j r_{(i_x\circ a^*)(f_j)}\circ (i_x\circ f_j^*)\circ r_a=\sum_i r_{a^*(e_ia)}\circ i_x\circ e_i^*,$$
which can be easily done as follows:
for all $a'\in A_{xy}$, we easily find that
$$
\Bigl(\sum_j r_{(i_x\circ a^*)(f_j)}\circ (i_x\circ f_j^*)\circ r_a\Bigr)(a')=a^*(a'a)=
\Bigl(\sum_i r_{a^*(e_ia)}\circ i_x\circ e_i^*\Bigr)(a').
$$
\equref{2.1.2} amounts to the equality of
\begin{eqnarray*}
&&\hspace*{-2cm}
\sigma_{xx}(m_{<0>})\ot_{B_x}m_{<1>}=
\sum_i\sigma_{xx}(m(i_x\circ e_i^*))\ot_{B_x} e_i\\
&=& \sum_{i,j} (m(i_x\circ e_i^*))(i_x\circ f_j^*)\ot_{B_x} f_j \ot_{B_x} e_i
\end{eqnarray*}
and
$$m_{<0>}\ot_{B_x} 1_x \ot_{B_x}m_{<1>}
= \sum_i m(i_x\circ e_i^*)\ot_{B_x} 1_x \ot_{B_x} e_i$$
in 
$M_{xx}\ot_{B_x}A_{xx}\ot_{B_x} A_{xy}~\congo{\equref{1.2.2}}~
\Hom_{B_x}(A_{xx}^*,M_{xx})\ot_{B_x} A_{xy}$. To this end, it suffices to show that
$$\sum_{j} (m(i_x\circ e_i^*))(i_x\circ f_j^*)\ot_{B_x} f_j =m(i_x\circ e_i^*)\ot_{B_x} 1_x$$
in $M_{xx}\ot_{B_x}A_{xx}\cong \Hom_{B_x}(A_{xx}^*,M_{xx})$, for all $i$. This is equivalent
to proving that
$$\sum_{j} ((m(i_x\circ e_i^*))(i_x\circ f_j^*))a^*( f_j) =(m(i_x\circ e_i^*))a^*( 1_x),$$
for all $a^*\in A_{xx}^*$, or
$$m\Bigl(\sum_i r_{(i_x\circ a^*)( f_j)}\circ i_x\circ f_j^*\circ i_x\circ e_i^*\Bigr)=
m\Bigl(r_{(i_x\circ a^*)( 1_x)}\circ i_x\circ e_i^*\Bigr),$$
so it suffices to show that
$$\sum_j r_{(i_x\circ a^*)( f_j)}\circ i_x\circ f_j^*\circ i_x\circ e_i^*=r_{(i_x\circ a^*)( 1_x)}\circ i_x\circ e_i^*.$$
Keeping in mind that $i_x(b)=b1_{xx}$ for all $b\in B_x$, we compute that
\begin{eqnarray*}
&&\hspace*{-2cm}
\sum_j (r_{(i_x\circ a^*)( f_j)}\circ i_x\circ f_j^*\circ i_x\circ e_i^*)(a)
=\sum_j f_j^*(e_i^*(a)1_x)a^*(f_j)1_x\\
&=& a^*(e_i^*(a)1_x)1_x=e_i^*(a)a^*(1_x)1_x
= (r_{(i_x\circ a^*)( 1_x)}\circ i_x\circ e_i^*)(a),
\end{eqnarray*}
for all $a\in A_{xy}$. Let us finally show that \equref{2.1.3} holds: for all $m\in M_{xy}$, we have that
$$m_{<0>}m_{<1>}=\sum_i (m(i_x\circ e_i^*))e_i=m\bigl(\sum_i r_{e_i}\circ i_x\circ e_i^*\bigr)=mA_{xy}=m,$$
where we used the fact that $\sum_i r_{e_i}\circ i_x\circ e_i^*=A_{xy}$. Indeed, for all $a\in A_{xy}$, we have
that
$$\sum_i (r_{e_i}\circ i_x\circ e_i^*)(a)=\sum_i e_i^*(a)e_i=a.$$
We define $K(M)=(M,\sigma)$ at the level of objects. We leave it to the reader to show that if $\varphi:\ M\to M'$
is right $\Aa$-linear, then $\varphi:\ (M,\sigma)\to (M',\sigma')$ is a morphism of descent data, so that we can
define $K$ as the identity at the level of morphisms.\\
It remains to be shown that $H\circ K$ and $K\circ H$ are the identity functors. Take a descent datum
$(M,\sigma)$ and let $KH(M,\sigma)=(M,\tilde{\sigma})$. We will show that $\tilde{\sigma}=\sigma$. For all $m\in M_{xy}$,
we have
\begin{eqnarray*}
\tilde{\sigma}_{xy}(m)&=&
\sum_i m((i_x\circ e_i^*)\ot_{B_x} e_i
= \sum_i m_{<0>}e_i^*(m_{<1>})\ot_{B_x} e_i\\
&=& \sum_i m_{<0>}\ot_{B_x} e_i^*(m_{<1>})e_i= m_{<0>}\ot m_{<1>}=\sigma_{xy}(m).
\end{eqnarray*}
Finally take a right $\Aa$-module $M$. Then $HK(M)=M$ with a new $\Aa$-action $\rightact$. We will show that it
coincides with the original one. For all $m\in M_{xy}$ and $f\in \Aa^x_{yz}={}_{B_x}\Hom(A_{xy},A_{xz})$, we have
that
$$m\rightact f =\sum_i(m(i_x\circ e_i^*))f(e_i)=m\bigl(\sum_i r_{f(e_i)}\circ i_x\circ e_i^*\bigr)=mf.$$
Here we used the fact that $\sum_i r_{f(e_i)}\circ i_x\circ e_i^*=f$, which can be seen as follows. For all $a\in A_{xy}$,
we have
$$\sum_i (r_{f(e_i)}\circ i_x\circ e_i^*)(a)=\sum_i e_i^*(a)f(e_i)=f(a).$$
\end{proof}

\begin{theorem}\thlabel{5.4}
Let $A$ and $B$ be as in \prref{5.1}. Then we have a pair of adjoint functors $(F_1,G_1)$
between the categories $\Dd_k(X)_B$ and $\Mm_\Aa$.  If $A$ is locally finite as a left
$B$-module, then the counit of this adjunction is an isomorphism.
If $A$ is locally faithfully projective as a left
$B$-module, then the unit of the adjunction is also an
isomorphism, and $(F_1,G_1)$ is a pair of inverse equivalences.
\end{theorem}

\begin{proof}
In the case where $A$ is locally finite as a left
$B$-module, the result is an immediate consequence of Propositions \ref{pr:2.3} and \ref{pr:2.4} and \thref{5.3}, taking into
account the remarks made in \seref{1.2}. Let us describe the functors $F_1=H\circ F$ and $G_1=G\circ K$.
First take $N\in \Dd_k(X)$. It is easy to show that $F_1(N)_{xy}= N_x\ot_{B_x} A_{xy}$, with right $\Aa$-action
$$(n\ot_{B_x} a)f= n\ot_{B_x} f(a),$$
for $n\in N_x$, $a\in A_{xy}$ and $f\in \Aa^x_{yz}={}_{B_x}\Hom(A_{xy}, A_{xz})$.\\
Now for $M\in \Mm_\Aa$, $K(M)=(M,\sigma)$ is defined in the proof of \thref{5.3}. Now
$$G_1(M)_x=G(M,\sigma)_x=\{m\in M_{xx}~|~\sigma_{xx}(m)=m\ot_{B_x} 1_x\}.$$
Now $\sigma_{xx}(m)=\sum_j m(i_x\circ f_j^*)\ot_{B_x} f_j$, where $\sum_j f_j^*\ot_{B_x} f_j$ is a dual basis
of $A_{xx}$ as a left $B_x$-module. Now $\sigma_{xx}(m)$ and $m\ot_{B_x} 1_x$ live in
$M_{xx}\ot_{B_x} A_{xx}\cong \Hom_{B_x} (A_{xx}^*, M_{xx})$, see \seref{1.2}. For all $a^*\in A^*_{xx}$, we have that
$$\sum_j  (m(i_x\circ f_j^*))a^*(f_j)=
\sum_j m\bigl(r_{i_x\circ a^*(f_j)}\circ i_x\circ f_j^*\bigr)
\equal{(*)} m(i_x\circ a^*).$$
At $(*)$, we used the following: for all $a\in A_{xx}$, we have that
$$r_{i_x\circ a^*(f_j)}\circ i_x\circ f_j^*= \sum_j f_j^*(a)1_xa^*(f_j)=a^*(a)1_x.$$
We conclude that $\sigma_{xx}(m)=m\ot_{B_x} 1_x$ if and only if $ma^*=ma^*(1_x)$ for all
$a^*\in A_{xx}^*$, and
$$G_1(M)_x=\{m\in M_{xx}~|~ma^*=ma^*(1_x)~{\rm for ~all~}a^*\in A_{xx}^*\}.$$
Now we drop the assumption that $A$ is locally finite. We can still define the functors
$F_1$ and $G_1$: the explicit formulas presented above do not involve the dual basis, and it can easily
be established that the same is true for the unit and the counit of the adjunction.
\end{proof}

Now consider the right endomorphism cluster $\Aa=\End_B(A)$.
We have a functor $H_1:\ {}_B\Desc(A)\to {}_\Aa\Mm$, which is an isomorphism
of categories if $A$ is locally finite as a right $B$-module. We have an adjunction $(F'_1,G'_1)$
between ${}_B\Dd_k(X)$ and ${}_\Aa\Mm$ which is a pair of inverse equivalences
if $A$ is locally faithfully projective as a right $B$-module.

\section{Dual $K$-Galois category extensions}\selabel{6}
We begin this Section with a new class of examples of clusters.
Let $K$ be a dual $k$-linear category. A right $K$-module category is a $k$-linear category $A$ such that every $A_{xy}$
is a right $K_{xy}$-module, and the following condition holds, for all
$a\in A_{xy}$, $a'\in A_{yz}$ and $k\in K_{xz}$:
\begin{equation}\eqlabel{1.1.2}
(aa')\cdot k= (a\cdot k_{(1,x,y)})(a'\cdot k_{(2,y,z)}).
\end{equation}
We have a cluster $\Bb=K\# A$,
$\Bb^x_{yz}= K_{xy}\# A_{yz}$.
The symbol $\#$ replaces $\ot$ and indicates that $\Bb^x$ is a $k$-linear category.
The multiplication and unit are given by the formulas
\begin{equation}\eqlabel{6.1.1}
(k\# a)(k'\# a')=kk'_{(1,x,y)}\# (a\cdot k'_{(2,y,z)})a'~~{\rm and}~~
1^x_y=1_{xy}\#1_y,
\end{equation}
for $k\in K_{xy}$, $k'\in K_{xz}$, $a\in A_{yz}$, $a'\in A_{zu}$. We call $K\# A$
the smash product cluster. Now we have a diagonal algebra $B=A^K$ defined as follows:
$$B_x= A_{xx}^{K_{xx}}=\{a\in A_{xx}~|~a\cdot k=\varepsilon_x(k)a,~~{\rm for~all}~k\in K_{xx}\}.$$
Moreover we have a canonical morphism of clusters
\begin{equation}\eqlabel{6.1.2}
\kappa:\ \Bb=K\# A\to \Aa={}_B\End(A)
\end{equation}
from the smash product cluster to the endocluster, given by the formulas
$$\kappa^x_{yz}:\ K_{xy}\# A_{yz}\to{}_{B_x}\Hom(A_{xy},A_{xz}),~~~
\kappa^x_{yz}(k\# a)(a')=(a'\cdot k)a.$$

\begin{definition}\delabel{6.1}
We call $A$ a dual (right) $K$-Galois category extension of $B=A^K$ if $\kappa$ is an isomorphism of clusters.
\end{definition}

\begin{proposition}\prlabel{6.2}
With notation as above,
we have an adjoint pair of functors $(F_2,G_2)$ between the categories $\Dd_k(X)_B$ and $\Mm_{\Bb}$.
Moreover $F_1=R\circ F_2$, where $F_1:\ \Dd_k(X)_B\to \Mm_\Aa$ is the functor defined in \thref{5.4}, and
$R:\ \Mm_\Bb\to \Mm_\Aa$ is the restriction of scalars functor via $\kappa$.\\
$(F_2,G_2)$ is a pair of inverse equivalences if the following conditions are satisfied
\begin{enumerate}
\item $A_{xy}$ is finitely generated and projective as a left $B_x$-module, for all $x,y\in X$;
\item $A_{xy}$ is a left $B_{xx}$-progenerator, for all $x\in X$;
\item $A$ is a dual $K$-Galois category extension of $B=A^K$.
\end{enumerate}
\end{proposition}

\begin{proof}
For $N\in \Dd_k(X)_B$, $F_2(N)_{xy}=N_x\ot_{B_x} A_{xy}$, with right $\Bb$-action defined as follows:
$$(n\ot a)(k\# a')=n\ot (a\cdot k)a',$$
for all $n\in N_x$, $a\in A_{xy}$, $k\in K_{xy}$ and $a'\in A_{yz}$.\\
For $M\in \Mm_\Bb$, $G_2(M)$ is defined as follows:
$$G_2(M)_x=\{m\in M_{xx}~|~m(k\#1_x)=\varepsilon_x(k)m,~{\rm for~all~}k\in K_x\}=M_{xx}^{K_{xx}}.$$
We have to show that $G_2(M)_x$ is a right $B_x$-module. First observe that $M$ is a right $A$-module
via restriction of scalars: the maps $M_{xy}\ot A_{yz}\to M_{xz}$ are defined as follows:
$$ma=m(1_{xy}\#a).$$
Now for $m\in M_{xx}$ and $b\in B_x\subset A_{xx}$, we define
$$mb=m(1_{xx}\#b).$$
This makes $M_{xx}$ a right $B_x$-module. $M_{xx}^{K_{xx}}$ is a $B_x$-submodule. Take $m\in M_{xx}^{K_{xx}}$
and $b\in B_x$. For all $k\in K_{xx}$, we have
\begin{eqnarray*}
(mb)(k\#1_x)&=&
m(1_{xx}\#b) (k\#1_x)=m(k_{(1,x,x)}\# b\cdot k_{(2,x,x)})\\
&=&m(k_{(1,x,x)}\# \varepsilon_x(k_{(2,x,x)})b)=m(k\# b)\\
&=& m(k\# 1_x)(1_{xx}\# b)=\varepsilon_x(k)mb.
\end{eqnarray*}
We describe the unit and the counit of the adjunction. For $M\in \Mm_\Bb$, $\varepsilon^M:\ F_2G_2(M)\to M$ has
the following components
$$\varepsilon^M_{xy}:\ M_{xx}^{K_{xx}}\ot_{B_x} A_{xy}\to M_{xy},~~\varepsilon^M_{xy}(m\ot_{B_x}a)=ma.$$
For $N\in \Dd_k(X)_B$, $\eta^N:\ N\to G_2F_2(N)$ has the following components:
$$\eta^N_x:\ N_x\to (N_x\ot_{B_x}A_{xx})^{K_{xx}},~~\eta^N_x(n)=n\ot_{B_x} 1_x.$$
Verification of the details is left to the reader. The second statement amounts to the following assertion.
Take $n\in N_x$, $a\in A_{xy}$, $k\in A_{xy}$ and $a'\in A_{yz}$. We need to show that
$(n\ot a)(k\#a')$ in $F_2(N)$ equals $(n\ot a)\kappa^x_{yz}(k\#a')$ in $F_1(N)$. Indeed,
$$(n\ot a)\kappa^x_{yz}(k\#a')=n\ot \kappa^x_{yz}(k\#a')(a)=n\ot (a\cdot k)a'=(n\ot a)(k\#a').$$
If $A$ is a dual $K$-Galois category extension of $B=A^K$, then $R$ is an isomorphism of categories. The
two other assumptions imply that $F_1$ is an equivalence of categories, see \thref{5.4}. The fact that $F_1=R\circ F_2$ implies that
$F_2$ is a category equivalence.
\end{proof}

Let us briefly state the left-handed versions of the results in this Section. For a left $K$-module category $A$,
we have a cluster $\Bb'=A\# K$, $\Bb'^x_{yz}=A_{yz}\#K_{zx}$, with multiplication
\begin{equation}\eqlabel{6.3.1}
(a\# k)(a'\# k')=a(k_{(1,z,u)}\cdot a')\# k_{(2,u,x)}k,
\end{equation}
for $a\in A_{yz}$, $k\in K_{zx}$, $a'\in A_{zu}$, $k'\in K_{ux}$. The units are $1^x_y=1_y\# 1_{yx}\in \Bb^x_{yy}$.
Let $B=A^K$. Then we have a canonical morphism of clusters 
\begin{equation}\eqlabel{6.3.2}
\kappa:\ \Bb'=A\# K\to \Aa'=\End_B(A),
\end{equation}
given by the formulas
$$\kappa^x_{yz}:\ A_{yz}\#K_{zx}\to \Hom_{B_x}(A_{zx},A_{yx}),~~\kappa^x_{yz}(a\# k)(a')=a(k\cdot a').$$
We call $A$ a dual (left) $K$-Galois category extension of $B=A^K$ if $\kappa$ is an isomorphism of clusters.

\section{Duality}\selabel{7}

\subsection{The Koppinen smash product}\selabel{7.1}
In the previous Sections, we have introduced the notions of $H$-Galois category extension and dual $K$-Galois category extension,
$H$ being a (semi)-Hopf category, and $K$ being is a dual (semi)-Hopf category. In this Section, we discuss how these notions are connected via duality.\\
To a right $H$-comodule category $A$, we associate a $k$-linear cluster $\Cc=\#(H,A)$, which can be
viewed as the multi-object version of the Koppinen
smash product \cite{K}.
$$\Cc^x_{yz}=\#(H_{xz},A_{zy}),$$
where $\#$ replaces $\Hom$, to indicate that we have specially defined multiplication maps
$\#:\ \Cc^x_{yz}\ot \Cc^x_{zu}\to \Cc^x_{yu}$ defined as follows:
for $g:\ H_{xz}\to A_{zy}$ and $g':\ H_{xu}\to A_{uz}$, we have
$g\# g':\ H_{xu}\to A_{uy}$ given by the formula
\begin{equation}\eqlabel{7.1.0}
(g\# g')(h)=g'(h_{(2)})_{[0]}g\bigl(h_{(1)}g'(h_{(2)})_{[1]}\bigr).
\end{equation}
The units are the maps $i^x_y=\eta_y\circ \varepsilon_{xy}:\ H_{xy}\to A_{yy}$,
$i^x_y(h)=\varepsilon_{xy}(h)1_y$. Verification of the associativity and unit conditions is
left to the reader.

\begin{proposition}\prlabel{7.1}
Let $A$ be a right $H$-comodule category, and let $B=A^{{\rm co}H}$. We have a morphism of clusters
$$\delta:\ \Cc=\#(H,A)\to \Aa^{\rm op}= {}_B\End(A)^{\rm op}.$$
$$\delta^x_{yz}:\ \Cc^x_{yz}=\#(H_{xz},A_{zy})\to \Aa^{{\rm op}x}_{yz}=\Aa^x_{zy}=
{}_{B_x}\Hom(A_{xz},A_{xy})$$
is given by the formula
\begin{equation}\eqlabel{7.1.1}
\delta^x_{yz}(g)(a)=a_{[0]}g(a_{[1]}).
\end{equation}
\end{proposition}

\begin{proof}
Let us first show that $\delta$ is multiplicative: for $g$ and $g'$ as above, we show that
$$\delta^x_{yu}(g\# g')= \delta^x_{yz}(g)\circ \delta^x_{zu}(g')
\in {}_{B_x}\Hom(A_{xu},A_{xy}).$$
For $a\in A_{xu}$, we have that
\begin{eqnarray*}
&&\hspace*{-10mm}
\bigl(\delta^x_{yz}(g)\circ \delta^x_{zu}(g')\bigr)(a)=
\delta^x_{yz}(g)\bigl(a_{[0]}g(a_{[1]})\bigr)
=a_{[0][0]}g(a_{[1]})_{[0]}g'\bigl(a_{[0][1]}g(a_{[1]})_{[1]}\bigr)\\
&=& a_{[0]}g(a_{[2]})_{[0]} g'\bigl( a_{[1]}g(a_{[2]})_{[1]}\bigr)
= a_{[0]}\bigl((g\# g')(a_{[1]})\bigr)=(\delta^x_{yu}(g\# g'))(a).
\end{eqnarray*}
$\delta$ preserves the units: $\delta^x_{yy}(i^x_y)$ is the identity map on $A_{xy}$
since
$\delta^x_{yy}(i^x_y)(a)=a_{[0]}i^x_y(a_{[1]})=a_{[0]}\varepsilon_{xy}(a_{[1]})1_y=a$,
for all $a\in A_{xy}$.
\end{proof}

\begin{proposition}\prlabel{7.2}
Assume that $A$ is an $H$-Galois category extension of $B=A^{{\rm co}H}$ (see \thref{3.5}).
 Then $\delta:\ \Cc=\#(H,A)\to \Aa^{\rm op}= {}_B\End(A)^{\rm op}$ is an isomorphism of clusters.
\end{proposition}

\begin{proof}
Consider the morphisms
$$\gamma_{xy}:\ H_{xy}\to A_{yx}\ot_{B_x} A_{xy},~~
\gamma_{xy}(h)=\sum_i l_i(h)\ot_{B_x} r_i(h),$$
from condition (3) in \thref{3.5}. We will show that
$$\tilde{\delta}^x_{yz}:\  {}_{B_x}\Hom(A_{xz},A_{xy})\to \#(H_{xz},A_{zy}),~~~
\tilde{\delta}^x_{yz}(\varphi)(h)=\sum_i l_i(h)\varphi(r_i(h))$$
is the inverse of ${\delta}^x_{yz}$. For all $\varphi\in {}_{B_x}\Hom(A_{xz},A_{xy})$,
$a\in A_{xz}$, $g\in \#(H_{xz},A_{zy})$ and $h\in H_{xz}$, we have that
\begin{eqnarray*}
&&\hspace*{-2cm}
({\delta}^x_{yz} \circ \tilde{\delta}^x_{yz})(\varphi)(a)=
a_{[0]}\bigl( \tilde{\delta}^x_{yz}(a_{[1]})\bigr)
= \sum_i a_{[0]} l_i(a_{[1]})\varphi(r_i(a_{[1]}))
\equal{\equref{3.5.2}} \varphi(a);\\
&&\hspace*{-2cm}
(\tilde{\delta}^x_{yz}\circ {\delta}^x_{yz})(g)(h)
= \sum_i l_i(h) {\delta}^x_{yz}(g)(r_i(h))
= \sum_i l_i(h)r_i(h)_{[0]}g(r_i(h)_{[1]})\\
&\equal{\equref{3.7.3}}&
\sum_i l_i(h_{(1)})r_i(h_{(1)})g(h_{(2)})
\equal{\equref{3.7.4}} \varepsilon_{xz}(h_{(1)})1_zg(h_{(2)})=g(h).
\end{eqnarray*}
\vspace*{-1cm}
\end{proof}

Our next goal is to prove the converse of \prref{7.2}. Some additional finiteness
conditions will be needed. 

\begin{proposition}\prlabel{7.3}
Let $A$ be a right $H$-comodule category and let $B=A^{{\rm co}H}$. Assume
that the following conditions are satisfied
\begin{enumerate}
\item $A$ locally finite as a left $B$-module;
\item $H$ is locally finite;
\item $\delta^x_{xy}$ and  $\delta^x_{yy}$ are bijective, for all $x,y\in X$.
\end{enumerate}
Then $A$ is an $H$-Galois category extension of $B=A^{{\rm co}H}$.
\end{proposition}

\begin{proof}
Let
$\sum_j a^*_j\ot_{B_x} a_j\in {}_{B_x}\Hom(A_{xy},B_x)\ot_{B_x} A_{xy}$
be a finite dual basis of $A_{xy}$ as a left $B_x$-module. We define
$\gamma_{xy}:\ H_{xy}\to A_{yx}\ot_{B_x} A_{xy}$ by the formula
\begin{equation}\eqlabel{7.3.1}
\gamma_{xy}(h)=\sum_i l_i(h)\ot_{B_x} r_i(h)=
\sum_j \tilde{\delta}^x_{xy}(i_x\circ a_j^*)(h)\ot_{B_x} a_j.
\end{equation}
Recall that $i_x:\ B_x\to A_{xx}$, so that $i_x\circ a_j^*\in {}_{B_x}\Hom(A_{xy},A_{xx})$.
We have to show that (\ref{eq:3.5.1}-\ref{eq:3.5.2}) are satisfied. \equref{3.5.2}
follows from the fact that $\tilde{\delta}^x_{xy}$ is a right inverse of ${\delta}^x_{xy}$:
for all $f\in {}_{B_x}\Hom(A_{xy},A_{xx})$ and $a\in A_{xy}$, we have
\begin{equation}\eqlabel{7.3.2}
f(a)=(({\delta}^x_{xy}\circ \tilde{\delta}^x_{xy})(f))(a)=a_{[0]}(\tilde{\delta}^x_{xy}(f))(a_{[1]}),
\end{equation}
hence
\begin{eqnarray*}
&&\hspace*{-2cm}
\sum_i a_{[0]}l_i(a_{[1]})\ot_{B_x}r_i(a_{[1]})\equal{\equref{7.3.1}}
\sum_{i,j}   a_{[0]} \tilde{\delta}^x_{xy}(i_x\circ a_j^*)(a_{[1]})\ot_{B_x} a_j\\
&\equal{\equref{7.3.2}}&
\sum_j a^*_j(a)1_x\ot_{B_x} a_j=1_x\ot_{B_x}a.
\end{eqnarray*}
Before we are able to prove \equref{3.5.1}, we need two observations. The first observation
is that we can reformulate the right coaction on $A$ in terms of the dual basis
$\sum_l h_l^*\ot h_l\in H^*_{xy}\ot H_{xy}$ of $H_{xy}$. For all $a\in A_{xy}$, we have
that
\begin{equation}\eqlabel{7.3.3}
\rho(a)=\sum_l a_{[0]}\lan h^*_l,a_{[1]}\ran \ot h_l
= \sum_l \delta^x_{yy}(\eta_y\circ h^*_l)(a)\ot h_l.
\end{equation}
The second observation is that $\delta$ is right $A$-linear in the following sense.
For $g\in \Hom(H_{xz},A_{zy})$, $\varphi\in {}_{B_x}\Hom(A_{xz},A_{xy})$ and $a'\in A_{yu}$,
we define $g\cdot a\in \Hom(H_{xz},A_{zu})$ and
$\varphi\cdot a'\in {}_{B_x}\Hom(A_{xz},A_{xu})$ by right multiplication:
$$(g\cdot a')(h)=g(h)a'~~~{\rm and}~~~(\varphi\cdot a')(a)=\varphi(a)a'.$$
It is then easily computed that
$$\delta^x_{uz}(g\cdot a')(a)=a_{[0]}(g\cdot a')(a_{[1]})=a_{[0]}g(a_{[1]})a'
=\delta^x_{yz}(g)(a)a'=(\delta^x_{yz}(g)\cdot a')(a).$$
Now $\delta^x_{xy}$ and $\delta^x_{yy}$ are invertible, so we have, for
$\varphi\in {}_{B_x}\Hom(A_{xy},A_{xx})$ and $a'\in A_{xy}$, that
\begin{equation}\eqlabel{7.3.4}
\tilde{\delta}^x_{yy}(\varphi\cdot a')=\tilde{\delta}^x_{xy}(\varphi)\cdot a'.
\end{equation}
For $h\in H_{xy}$, we compute that
\begin{eqnarray*}
L&:=&\sum_i l_i(h)r_i(h)_{[0]}\ot r_i(h)_{[1]}\\
&\equal{\equref{7.3.1}}&
\tilde{\delta}^x_{xy}(i_x\circ a_j^*)(h) a_{j[0]}\ot a_{j[1]}\\
&\equal{\equref{7.3.3}}&
\sum_{j,l} \Bigl(\tilde{\delta}^x_{xy}(i_x\circ a_j^*)(h)\Bigr)
\Bigl(\delta^x_{yy}(\eta_y\circ h^*_l)(a_j)\Bigr)\ot h_l
\end{eqnarray*}
Now let $a'= \delta^x_{yy}(\eta_y\circ h^*_l)(a_j)\in A_{xy}$. Applying
\equref{7.3.4}, we obtain that
$$L=\sum_{j,l} \tilde{\delta}^x_{yy}\bigl((i_x\circ a_j^*)\cdot a'\bigr)(h)\ot h_l.$$
For all $a''\in A_{xy}$, we have that
$$\sum_j\bigl((i_x\circ a_j^*)\cdot a'\bigr)(a'')
=\sum_j a^*_j(a'')1_x \delta^x_{yy}(\eta_y\circ h^*_l)(a_j)=\delta^x_{yy}(\eta_y\circ h^*_l)(a''),$$
and
\begin{eqnarray*}
L&=& 
\sum_l \Bigl(\tilde{\delta}^x_{yy}\bigl(\delta^x_{yy}(\eta_y\circ h^*_l)\bigr)\Bigr)(h)\ot h_l\\
&=&\sum_l (\eta_y\circ h_l^*)(h)\ot h_l=\sum_l \lan h^*_l,h\ran 1_y\ot h_l=1_y\ot h.
\end{eqnarray*}
completing the proof of \equref{3.5.1}.
\end{proof}

\subsection{The smash product versus the Koppinen smash product}\selabel{7.2}
We briefly return to the classical situation, where $X$ is a singleton. Consider a
(finitely generated projective) bialgebra $H$ coacting from the right on a $k$-module $M$. 
The usual way to define an action of the dual bialgebra $H^*$ is via the formula
\begin{equation}\eqlabel{action}
h^*\cdot m=\lan h^*,m_{[1]}\ran m_{[0]}.
\end{equation}
If $M=A$ is a right $H$-comodule algebra, then $A$ is a left $H^*$-module algebra.
Here we need the convolution product and the convolution coproduct on $H^*$.\\
We can also consider the formula 
\begin{equation}\eqlabel{action2}
m\cdot h^*=\lan h^*,m_{[1]}\ran m_{[0]}.
\end{equation}
It makes
$A$ into a right $H^*$-module algebra, now $H^*$ equipped with anti-convolution
product, but convolution coproduct.\\
Passing to the categorical situation where $X$ is no longer a singleton, \equref{action}
no longer makes sense. There are two ways to fix this problem; one may use the
antipode (if it exists), we come back to this in \reref{7.8}. An alternative solution
is to introduce some op-arguments. From now on, $H$ is a locally finite semi-Hopf
category, with corresponding dual semi-Hopf category $K$, as in \seref{1.3}.
The proof of \prref{7.4} is a direct verification.

\begin{proposition}\prlabel{7.4}
Let $A$ be a right $H$-module category. Then 
$A^{\rm op}$ is a left $K^{\rm op}$-module category, with
\begin{equation}\eqlabel{7.5.1}
k\cdot a=\lan k, a_{[1]}\ran a_{[0]},
\end{equation}
for all $k\in K_{xy}$ and $a\in A_{xy}^{\rm op}=A_{yx}$.
\end{proposition}

The associated smash product
$\Bb'=A^{\rm op}\#K^{\rm op}$ is described as follows, see \equref{6.3.1}:
$\Bb'^{x}_{yz}=A_{zy}\#K_{zx}$, with multiplication ($a'\in A_{uz}$, $k'\in K_{ux}$)
\begin{equation}\eqlabel{7.5.2}
(a\# k)(a'\# k')=  \lan k_{(1,z,u)},a'_{[1]}\ran a'_{[0]}a\# k' k_{(2,u,x)}.
\end{equation}
Observe that $\End_{B^{\rm op}}(A^{\rm op})={}_B\End(A)^{\rm op}=\Aa^{\rm op}$. Indeed,
$$\Hom_{B_x}^{\rm op}(A^{\rm op}_{zx}, A^{\rm op}_{yx})=
{}_{B_x}\Hom(A_{xz},A_{xy})=\Aa^x_{zy}=\Aa^{{\rm op}x}_{yz}.$$
Applying \equref{6.3.2}, we obtain a morphism of clusters $\kappa':\ \Bb'\to \Aa^{\rm op}$,
\begin{equation}\eqlabel{7.5.2a}
\kappa'^{x}_{yz}:\ A_{zy}\# K_{zx}\to {}_{B_x}\Hom(A_{xz},A_{xy}),~~
\kappa'^{x}_{yz}(a\# k)(a')=(k\cdot a')a.
\end{equation}

\begin{proposition}\prlabel{7.6}
Let $H$ be a locally finite $k$-linear semi-Hopf category, and let $A$ be a right $H$-comodule category.
We have an isomorphism of clusters $\beta:\ \Bb'=A^{\rm op}\# K^{\rm op}\to \Cc=\Hom(H,A)$,which is
such that the diagram
$$\xymatrix{\Bb'=A^{\rm op}\# K^{\rm op}\ar[d]_{\beta}\ar[rr]^{\kappa'}&&
\Aa^{\rm op}={}_B\End(A)^{\rm op}\\
\Cc=\#(H,A)\ar[rru]_{\delta}&&}$$
commutes.
\end{proposition}

\begin{proof}
It is well-known that
$\beta^x_{yz}:\ A_{zy}\# K_{zx}\to \Hom(H_{xz},A_{zy})$ and $\tilde{\beta}^x_{yz}:\ \Hom(H_{xz},A_{zy})
\to A_{zy}\# K_{zx}$ given by the formulas
\begin{equation}\eqlabel{7.6.1}
\beta^x_{yz}(a\# k)(h)=\lan k,h\ran a ~~{\rm and}~~\tilde{\beta}^x_{yz}(f)=\sum_i f(h_i)\ot k_i,
\end{equation}
where $\sum_i h_i\ot k_i$ is the dual basis of $H_{xz}$, are inverses. It is clear that $\beta^x_{yy}(1^A_y\# \varepsilon_{xz})=
\eta_y\circ \varepsilon_{xz}$. Let us show that $\beta$ preserves the mulitplication.
Take $a\# k\in A_{zy}\# K_{zx}$, $a'\# k'\in A_{uz}\# K_{ux}$, and write
$\beta^x_{yz}(a\# k)=g$, $\beta^x_{zu}(a'\# k')=g'$. For all $h\in H_{xu}$, we have that
\begin{eqnarray*}
&&\hspace*{-2cm}
\beta^x_{yu}\bigl((a\# k)(a'\# k')\bigr)(h)\equal{(\ref{eq:7.5.2},\ref{eq:7.6.1})}
\lan k_{(1,z,u)},a'_{[1]}\ran \lan k'k_{(2,u,x)},h\ran a'a_{[0]}\\
&=& \lan k_{(1,z,u)},a'_{[1]}\ran\lan k',h_{(2)}\ran \lan k_{(2,u,x)},h_{(1)}\ran a'a_{[0]}\\
&=& \lan k, h_{(1)}a'_{[1]}\ran\lan k',h_{(2)}\ran a'a_{[0]}\\
&\equal{\equref{7.6.1}}& g'(h_{(2)})_{[0]}g\bigl(h_{(1)}g'(h_{(2)})_{[1]}\bigr)
\equal{\equref{7.1.0}}(g\# g')(h).
\end{eqnarray*}
We are left to show that $\delta^x_{yz}\circ \beta^x_{yz}=\kappa'^{x}_{yz}$. For all $a'\in A_{xz}$, we have that
\begin{eqnarray*}
&&\hspace*{-2cm}
\delta^x_{yz}\bigl(\beta^x_{yz}(a\# k)\bigr)(a')\equal{\equref{7.1.1}} a'_{[0]}\bigl(\beta^x_{yz}(a\# k)(a'_{[1]})\bigr)\\
&\equal{\equref{7.6.1}}& \lan k,a'_{[1]}\ran a'_{[0]}a\equal{\equref{7.5.2a}}\kappa'^{x}_{yz} (a\# k)(a').
\end{eqnarray*}
\end{proof}

We now summarize our results.

\begin{theorem}\thlabel{7.7}
Let $H$ be a locally finite semi-Hopf category, with dual $K=H^*$
and let $A$ be a right $H$-comodule category. Then $A^{\rm op}$ is a right $K^{{\rm op}}$-module category,
see \equref{7.5.1}, and  $A^{K^{\rm op}}=A^{{\rm co}H}=B$.
The following assertions are equivalent.
\begin{enumerate}
\item $A$ is an $H$-Galois category extension of $B$, that is, 
$\can^z_{xy}:\ A_{zx}\ot_{B_x} A_{xy}\to A_{zy}\ot H_{xy}$ is bijective, for all $x,y,z\in X$,
see \thref{3.5};
\item $A^{\rm op}$ is a dual left $K^{\rm op}$-Galois category extension of $B=A^{K^{\rm op}}$, that is,
$\kappa':\ \Bb'=A^{\rm op}\# K^{\rm op}\to \Aa^{\rm op}={}_B\End(A)^{\rm op}$ is an isomorphism of $k$-linear clusters;
\item $\delta:\ \Cc\to \#(H,A)\to \Aa^{\rm op}={}_B\End(A)^{\rm op}$ is an isomorphism of $k$-linear clusters;
\item $\kappa'^{x}_{yx}$ and $\kappa'^{x}_{yy}$ are bijective, for all $x,y\in X$;
\item $\delta^x_{yx}$ and $\delta^{x}_{yy}$ are bijective, for all $x,y\in X$.
\end{enumerate}
\end{theorem}

\begin{proof}
$\ul{(1)\Rightarrow (3)}$: \prref{7.2}; $\ul{(3)\Rightarrow (5)}$ is trivial; $\ul{(5)\Rightarrow (1)}$: \prref{7.3};
$\ul{(2)\Leftrightarrow (3)}$ and $\ul{(4)\Leftrightarrow (5)}$: \prref{7.6}.
\end{proof}

\begin{remark}\relabel{7.8}
In the preceding Sections, we have provided a detailed account of the righthanded theory, ending with a brief
description of the lefthanded theory at the end of each Section. It may come as a surprise that we have to switch
from right to left in \thref{7.7}: in order to give alternative characterizations of $A$ being an $H$-Galois category
extension of $B$, we need a smash product obtained from a left action. If $H$ has an antipode, then we can
also work with a right action:  $A$ is a right $K^{\rm op}$-module category, with action is given by the formula
\begin{equation}\eqlabel{7.5.5}
a\rightact k=\lan k, S_{xy}(a_{[1]})\ran a_{[0]},
\end{equation}
for all $k\in K_{xy}$ and $a\in A_{xy}$.
The associated smash product
$\ol{\Bb}= K^{\rm op}\# A$ is described as follows:
$\ol{\Bb}^x_{yz}= K_{xy}\# A_{yz}$, with multiplication ($k'\in K_{xz}$, $a'\in A_{zu}$)
\begin{equation}\eqlabel{7.5.6}
(k\# a)(k'\# a')=k'_{(1,x,y)}k\# (a\rightact k'_{(2,y,z)})a'.
\end{equation}
Applying \equref{6.1.2}, we have
a morphism of clusters $\ol{\kappa}:\ \ol{\Bb}\to {}_B\End(A)=\Aa$,
$$\ol{\kappa}^x_{yz}:\ K_{xy}\# A_{yz}\to {A}^x_{yz}={}_{B_x}\Hom(A_{xy},A_{xz}),~~
\ol{\kappa}^x_{yz}(k\# a)(a')=(a'\rightact k)a.$$
In \prref{7.9}, we will see that this brings nothing new: the smash products $\ol{\Bb}$ and
$\Bb'$ are anti-isomorphic.
\end{remark}

\begin{proposition}\prlabel{7.9}
Let $H$ be a locally finite $k$-linear Hopf category, and let $A$ be a right $H$-comodule category.
We have an isomorphism of clusters
$\alpha:\ \ol{\Bb}=K^{\rm op}\# A\to \Bb'^{\rm op}=(A^{\rm op}\# K^{\rm op})^{\rm op}$
which is such that the diagram
$$\xymatrix{
\ol{\Bb}\ar[d]_{\alpha}\ar[rr]^{\ol{\kappa}}&&\Aa\\
\Bb'^{\rm op}\ar[urr]_{\kappa'^{{\rm op}}}&&}$$
commutes. Consequently the equivalent statements of \thref{7.7} are also equivalent to
\begin{enumerate}
\item[(6)] $A$ is a dual right $K^{\rm op}$-Galois category extension of $B$, that is,
$\ol{\kappa}:\ \ol{\Bb}\to {}_B\End(A)=\Aa$ is an isomorphism of $k$-linear clusters;
\item[(7)] $\ol{\kappa}^{x}_{xy}$ and $\ol{\kappa}^{x}_{yy}$ are bijective, for all $x,y\in X$.
\end{enumerate}
\end{proposition}

\begin{proof}
$\alpha$ is defined as follows:
$$
\alpha^x_{yz}:\ \ol{\Bb}^x_{yz}=K_{xy}\#A_{yz}\to \Bb'^{x}_{zy}= A_{yz}\# K_{yx},~~
\alpha^x_{yz}(k\# a)=a\# S_{xy}(k).$$
From the fact that
$H$ is locally finite (see \cite[Prop. 10.6]{BCV}), it follows that the antipode of $H$ is bijective, and this implies that
$\alpha$ is bijective, with inverse given by the formula
$$\bigl(\tilde{\alpha}^x_{yz}\bigr)(a\# k)=S^{-1}_{xy}(k)\# a.$$
It is left to the reader to show that $\alpha$ is multiplicative. 
 and $\tilde{\alpha}$ are bijective. Let us verify that the first triangle commutes. For all
$x,y,z\in X$ the triangle
\begin{equation}\eqlabel{7.5.9}
\xymatrix{\ol{\Bb}^{x}_{yz}\ar[d]_{\alpha^{x}_{yz}}\ar[rr]^{\ol{\gamma}^{x}_{yz}}&&
\Aa^x_{yz}\\
\Bb'^{x}_{zy}\ar[urr]_{\kappa'^{x}_{zy}}&&}
\end{equation}
commutes: for all $k\in K_{xy}$, $a\in A_{yz}$ and $a'\in A_{xy}$, we have that
\begin{eqnarray*}
&&\hspace*{-2cm}
(\kappa'^{x}_{zy}\circ \alpha^{x}_{yz})(k\# a)(a')=
\kappa'^{x}_{zy}\bigl(a\# S_{xy}(k)\bigr)(a')\\
&=&(S_{xy}(k)\cdot a')a=(a'\rightact k)a= \ol{\kappa}^{x}_{yz}(k\# a)(a').
\end{eqnarray*}
\end{proof}

\subsection{The Koppinen smash product revisited}\selabel{7.3}
In \seref{7.1}, we introduced the Koppinen smash product, and gave its relationship to right
faithfully projective descent data. Now we present an alternative version, related to left
faithfully projective descent data. This Koppinen smash product is isomorphic to a smash
product associated to a right $H$-module category in the sense of \seref{6}. We restrict to
giving the main results, the proofs are similar to the proofs presented in Sections \ref{se:7.1} and
\ref{se:7.2}. We assume that $H$ is locally finite semi-Hopf category with associated dual
semi-Hopf category $K$, and that $A$ is a right $H$-comodule category.\\
We have a $k$-linear cluster $\Cc'=\#'(H,A)$, defined componentwise as
$$\Cc'^x_{yz}=\#'(H_{yx},A_{zy}),$$
with the following multiplication: for $g\in \#'(H_{yx},A_{zy})$ and $g'\in \#'(H_{zx},A_uz)$,
$g\#'g'\in \#'(H_{yx},A_uy)$ is given by the formula
\begin{equation}\eqlabel{7.10.1}
(g\#'g')(h)=g'\bigl(g(h_{(2)})_{[1]}h_{(1)}\bigr)g(h_{(2)})_{[0]},
\end{equation}
for $h\in H_{yx}$. Our next observation is that \equref{action2} can be applied to construct a right
$K$-module category: $A^{\rm op}$ is a right $K$-module category, with action given by the formula
\begin{equation}\eqlabel{7.5.3}
a\cdot k=\lan k, a_{[1]}\ran a_{[0]},
\end{equation}
for all $k\in K_{xy}$ and $a\in A_{xy}^{\rm op}=A_{yx}$. The associated smash product
$\Bb= K\# A^{\rm op}$ is described as follows:
$\Bb^{x}_{yz}= K_{xy}\# A_{zy}$, with multiplication 
($a'\in A_{uz}$, $k'\in K_{xz}$)
\begin{equation}\eqlabel{7.5.4}
(k\# a)(k'\# a')= kk'_{(1,x,y)}\# a'(a\cdot k_{(2,y,z)})=\lan k_{(2,y,z)}, a_{[1]}\ran
kk'_{(1,x,y)}\# a'a_{[0]}.
\end{equation}
Now ${}_{B^{\rm op}}\End(A^{\rm op})=\End_B(A)^{\rm op}={\Aa'}^{\rm op}$. Applying \equref{6.1.2}, we have
a morphism of clusters $\kappa:\ \Bb\to \End_B(A)^{\rm op}=\Aa'^{\rm op}$,
$$\kappa^{x}_{yz}:\ \Bb^x_{yz}=K_{xy}\# A_{zy}\to \Aa'^x_{zy}=\Hom_{B_x}(A_{yx},A_{zx}),~~
\kappa^{x}_{yz}(k\# a)(a')=a(a'\cdot k).$$

\begin{theorem}\thlabel{7.11}
Let $H$ be a locally finite $k$-linear semi-Hopf category, and let $A$ be a right $H$-comodule category.
We have an isomorphism of clusters $\beta':\ \Bb\to \Cc'$ and a morphism of clusters $\delta':\ \Cc'\to 
\Aa'^{\rm op}$ 
such that the diagram
$$\xymatrix{\Bb=K\# A^{\rm op}\ar[d]_{\beta'}\ar[rr]^{\kappa}&&
\Aa^{\rm op}=\End_B(A)^{\rm op}\\
\Cc'=\#'(H,A)\ar[rru]_{\delta'}&&}$$
commutes. Also assume that $A$ is locally finite as a right $B$-module. The following assertions are equivalent:
\begin{enumerate}
\item $A$ is an $H$-Galois' category extension of $B=A^{{\rm co}H}$;
\item $A$ is a dual right $K$-Galois category extension of $B$, that is, $\kappa$
is an isomorphism of clusters;
\item $\delta'$ is an isomorphism of clusters;
\item $\kappa^{x}_{yx}$ and $\kappa^{x}_{yy}$ are invertible, for all $x,y\in X$;
\item $\delta'^{x}_{yx}$ and $\delta'^{x}_{yy}$ are invertible, for all $x,y\in X$.
\end{enumerate}
\end{theorem}

\begin{proof} (sketch)
$\delta'^x_{yz}:\ \Cc'^x_{yz}=\#'(H_{yx},A_{zy})\to \Aa'^x_{zy}=\Hom_{B_x}(A_{yx},A_{zx})$ is given by
the formula
$$\delta'^x_{yz}(g)(a')=g(a'_{[1]})a'_{[0]},$$
for all $a'\in A_{yx}$. $\beta'^x_{yz}:\ K_{xy}\# A_{zy}\to \#'(H_{yx},A_{zy})$ is given by the formula
$\beta'^x_{yz}(k\# a)(h)=\lan k,h\ran a$.\\
$\ul{(1)\Rightarrow (3)}$. It follows from \thref{3.12} that there exists $\gamma'_{zx}:\ H_{zx}\to A_{zx}\ot_{B_x} A_{xz}$
satisfying (\ref{eq:3.11.1}-\ref{eq:3.11.2}). The inverse of $\delta'^x_{yz}$ is given by the formula
(with notation as in \thref{3.12}):
$\bigl(\delta'^x_{yz}\bigr)^{-1}(\varphi)(h)=\sum_i \varphi(l'_i(h)) r'_i(h)$.
$\ul{(5)\Rightarrow (1)}$. Let
$\sum_j a_j\ot_{B_x} a^*_j\in A_{zx}\ot_{B_x} \Hom_{B_x}(A_{zx},B_x)$
be the finite dual basis of $A_{zx}$ as a right $B_x$-module. Observe that $i_x\circ a^*_j\in \Hom_{B_x}(A_{zx},A_{xx})$,
and define $\gamma'_{zx}$ as follows:
$\gamma'_{zx}(h)= \sum_j a_j\ot_{B_x} \bigl(\delta'^x_{zx}\bigr)^{-1}(i_x\circ a_j^*)(h)$,
for $h\in H_{zx}$. The proof of \equref{3.11.2} is straightforward. \equref{3.11.1} is more tricky, and depends
on the assumption that $\delta'^x_{zz}$ is invertible. We sketch the details.
For $a'\in A_{uz}$, $g\in \#'(H_{yx}, A_{zy})$ and $\varphi\in \Hom_{B_x}(A_{yx}, A_{zy})$, we define
$a'\cdot g\in \#'(H_{yx}, A_{uy})$ and $a'\cdot \varphi\in \Hom_{B_x}(A_{yx}, A_{uy})$ by left multiplication:
$(a'\cdot g)(h)=a' g(h)$ and $(a'\varphi)(a)=a'\varphi(a)$. Obviously $\delta'$ is left $A$-linear, in the sense that
$\delta'^x_{yu}(a'\cdot g)=a'\cdot \delta'^x_{yz}(g)$. This implies that
\begin{equation}\eqlabel{7.11.1}
\bigl(\delta'^x_{zz}\bigr)^{-1}(a'\cdot \varphi)=a'\cdot \bigl(\delta'^x_{zx}\bigr)^{-1}(\varphi),
\end{equation}
for $\varphi\in \Hom_{B_x}(A_{zx},A_{xx})$ and $a'\in A_{zx}$. Let $\sum_j h^*_l\ot h_l$ be the finite dual basis of $H_{zx}$.
Then we have for all $a\in A_{zx}$ that 
\begin{equation}\eqlabel{7.11.2}
\rho_{zx}(a)=\sum_l \delta'^x_{zz}(\eta_z\circ h^*_l)(a)\ot h_l=\sum_l a^l\ot h_l.
\end{equation}
With this notation, we can show that
\begin{equation}\eqlabel{7.11.3}
\sum_j a_j^l\cdot (i_x\circ a^*_j)=\delta'^x_{zz}(\eta_z\circ h^*_l).
\end{equation}
Finally
\begin{eqnarray*}
&&\hspace*{-2cm}
\sum_i l'_i(h)_{[0]}r'_i(h)\ot l'_i(h)_{[1]}
= \sum_j a_{j[0]} \bigl(\delta'^x_{zx}\bigr)^{-1}(i_x\circ a_j^*)(h)\ot a_{j[1]}\\
&\equal{\equref{7.11.2}}&
\sum_{j,l}  \Bigl(a_j^l \cdot \bigl(\bigl(\delta'^x_{zx}\bigr)^{-1}(i_x\circ a_j^*)\bigr)\Bigr)(h)\ot h_l\\
&\equal{\equref{7.11.1}}&
\sum_{j,l} \Bigr(\bigl(\delta'^x_{zx}\bigr)^{-1}(a_j^l\cdot (i_x\circ a^*_j))\Bigr)(h)\ot h_l\\
&\equal{\equref{7.11.3}}&
\sum_{l} \Bigr(\bigl(\delta'^x_{zz}\bigr)^{-1}(\delta'^x_{zz}(\eta_z\circ h^*_l))\Bigr)(h)\ot h_l\\
&=& \sum_l \lan h^*_l,h\ran 1_z\ot h_l=1_z\ot h,
\end{eqnarray*}
proving \equref{3.11.1}.
\end{proof}

\begin{remark}\relabel{7.12}
If $H$ is a Hopf category, then $A$ is a left $K$-module category, with action
$$k\leftact a=\lan k, S_{xy}(a_{[1]})\ran a_{[0]},$$
for $k\in K_{xy}$ and $a\in A_{xy}$. Proceeding as in \reref{7.8}, we can add two more
equivalent conditions ito \thref{7.11}. Moreover, the conditions in Theorems \ref{th:7.7}
and \ref{th:7.11} are equivalent, by \thref{3.14}.
\end{remark}

\section{Hopf-Galois extensions and groupoid graded algebras}\selabel{8}
Let $G$ be a groupoid, with underlying class of objects $X$. The unit element of $G_{xx}$ is
denoted as $e_x$. Then $kG$ is a $k$-linear Hopf category, see \cite[Ex. 3.4]{BCV}.
In the situation where $X$ is a set, we can consider the groupoid algebra $kG$, which is the
Hopf category $kG$ in packed form: $kG=\oplus_{x,y\in X} kG_{xy}$, with multiplication
extended linearly from the composition in $G$, where we put $\tau\sigma=0$ if $\tau$ and
$\sigma$ cannot be composed. If $X$ is finite, then $kG$ has the unit $\sum_x e_x$.\\
A $G$-grading on $M\in \Mm_k(X)$ consists of a direct sum decomposition
$$M_{xy}=\oplus_{\sigma\in G_{xy}} M_\sigma,$$
for all $x,y\in X$. If $m\in M_\sigma$, then $m$ is said to be homogeneous of degree $\sigma$,
written as ${\rm deg}(m)=\sigma$. $\Mm_k(X)^G$ is the category of $G$-graded objects of
$\Mm_k(X)$. Its morphisms are degree preserving morphisms in $\Mm_k(X)$.\\
A $G$-graded $k$-linear category is a $k$-linear category $A$ with a $G$-grading such that
$1_x\in A_{e_x}$ and $A_{\sigma}A_{\tau}\subset A_{\sigma\tau}$, for all $x,y,z\in X$
and $\sigma\in G_{xy}$ and $\tau\in G_{yz}$. If $A_{\sigma}A_{\tau}= A_{\sigma\tau}$
for all $\sigma$ and $\tau$, then $A$ is called a strongly $G$-graded $k$-linear category.
If $X$ is a (finite) set, then these definitions can be restated in packed form, and we recover
definitions from \cite{L}, where a structure theorem for strongly graded algebras over a
groupoid is presented.\\
A $G$-graded right $A$-module is an object $M\in \Mm_k(X)^G$ with a right $A$-action such that
$M_{\sigma}A_{\tau}\subset M_{\sigma\tau}$, for all $x,y,z\in X$
and $\sigma\in G_{xy}$ and $\tau\in G_{yz}$. $\Mm_k(X)^G_A$ is the category of 
$G$-graded right $A$-modules.

\begin{proposition}\prlabel{8.1}
For a groupoid $G$, the categories $\Mm_k(X)^{G}$ and $\Mm_k(X)^{kG}$ are isomorphic.
$kG$-comodule category structures on a $k$-linear category $A$ correspond bijectively to
$G$-gradings on $A$, and, in this situation, the categories 
$\Mm_k(X)^{G}_A$ and $\Mm_k(X)^{kG}_A$ are isomorphic.
\end{proposition}

\begin{proof}
For $M\in \Mm_k(X)^{G}$, the maps $\rho_{xy}:\ M_{xy}\to M_{xy}\ot kG_{xy}$
given by the formula $\rho_{xy}(m)=m\ot\sigma$ if ${\rm deg}(m)=\sigma$, extended linearly,
define a right $kG$-coaction on $M$.\\
For  $M\in \Mm_k(X)^{kG}$ and $\sigma\in G_{xy}$, define
$$M_\sigma=\{m\in M_{xy}~|~\rho_{xy}(m)=m\ot\sigma\}.$$
It is clear that $M_\sigma\cap M_\tau=\{0\}$ if $\sigma\neq\tau\in G_{xy}$.
Since $\rho_{xy}:\ M_{xy}\to M_{xy}\ot k G_{xy}=\oplus_{\sigma\in G_{xy}}M_{xy}\ot\sigma$, we can write
$$\rho_{xy}(m)=\sum_{\sigma\in G_{xy}} m_\sigma\ot\sigma,$$
with $m_\sigma\in M_{xy}$. Applying $M_{xy}\ot \varepsilon_{xy}$ to both sides, we see that
$$m= \sum_{\sigma\in G_{xy}} m_\sigma.$$
The coassociativity of $\rho$ entails that
$$\sum_{\sigma\in G_{xy}}\rho_{xy}(m_\sigma)\ot \sigma=\sum_{\sigma\in G_{xy}} m_\sigma \ot \sigma
\ot \sigma \in \oplus_{\sigma\in G_{xy}} M_{xy}\ot kG_{xy}\ot k\sigma.$$
Fixing $\tau\in G_{xy}$, and taking the projection of both sides onto the component $M_{xy}\ot kG_{xy}\ot k\tau$,
we find that $\rho_{xy}(m_\tau)=m_\tau\ot \tau$, and $m_\tau\in M_\tau$, for all $\tau\in G_{xy}$. This proves
that $m= \sum_{\sigma\in G_{xy}} m_\sigma\in \oplus_{\sigma\in G_{xy}} M_\sigma$.\\
The proof of other assertions is similar and is left to the reader.
\end{proof}

\begin{theorem}\thlabel{8.2}
Let $G$ be a groupoid, and
let $A$ be a $G$-graded $k$-linear category. Let $B=A^{{\rm co}kG}$ be the diagonal algebra with $B_x=A_{e_x}$.
The following statements are equivalent.
\begin{enumerate}
\item[(1)] $A$ is strongly graded;
\item[(2)] $A_{\sigma^{-1}}A_{\sigma}=B_y$, for all $x,y\in X$ and $\sigma\in G_{xy}$;
\item[(3)] The adjunction $(F,G)$ from \prref{3.2} is a pair of inverse equivalences;
\item[(4)] $A$ is a $kG$-Galois category extension of $B=A^{{\rm co}kG}$, in the sense of \thref{3.5}.
\end{enumerate}
\end{theorem}

\begin{proof}
$\ul{(1)\Rightarrow (2)}$ is obvious.\\
$\ul{(2)\Rightarrow (3)}$. Take $N\in \Dd_k(X)_B$ and $M\in \Mm_k(X)_A^G$. Recall from \prref{3.2} that
$F(N)_{xy}=N_x\ot_{B_x} A_{xy}$ and $G(M)_x=M^{{\rm co}kG_{xx}}_xx=M_{e_x}$. It is obvious that
$$\eta^N_x:\ N_x\to F(N)_{e_x}= N_x\ot_{B_x}A_{e_x},~~\eta^N_x(n)=n\ot_{B_x}1_x$$
is bijective, for all $x\in X$. We are done if we can show that
$$\varepsilon^M_{xy}:\ M_{e_x}\ot_{B_x}A_{xy}\to M_{xy},~~\varepsilon^M_{xy}(m\ot_{B_x} a)=ma$$
is bijective. Take $\sigma\in G_{xy}$ and $m\in M_{\sigma}\subset M_{xy}$. (2) implies that there exist
$a_i\in A_{\sigma^{-1}}\subset A_{yx}$ and $a_i'\in A_{\sigma}\subset A_{xy}$ such that $\sum_i a_ia'_i=1_y$.
Then $ma_i\in M_{\sigma}A_{\sigma^{-1}}\subset M_{e_x}$, $\sum_i ma_i\ot_{B_x} a'_i\in M_{e_x}\ot_{B_x} A_{xy}$,
and $\varepsilon^M_{xy}(\sum_i ma_i\ot_{B_x} a'_i)=m$. This proves that $\varepsilon^M_{xy}$ is surjective.\\
Finally take $\omega=\sum_j m_j\ot_{B_x} c_j\in \Ker(\varepsilon^M_{xy})\subset M_{e_x}\ot_{B_x} A_{xy}$. For
each $j$, we have that
$$c_j=\sum_{\sigma\in G_{xy}}c_{j\sigma}\in \oplus_{\sigma\in G_{xy}}A_\sigma.$$
For all $\sigma\in G_{xy}$, we find that $0=\varepsilon^M_{xy}(\omega)_\sigma=\sum_j m_jc_{j\sigma}$. Using the fact
that $c_{j\sigma}a_i\in A_\sigma A_{\sigma^{-1}}=B_x$, we find that
$$\omega_\sigma=\sum_j m_j\ot_{B_x} c_{j\sigma}=\sum_{i,j} m_j\ot_{B_x} c_{j\sigma}a_ia'_i
= \sum_{i,j} m_j\ c_{j\sigma}a_i\ot_{B_x}a'_i=0.$$
It follows that $\omega=0$, and this shows that $\varepsilon^M_{xy}$ is injective.\\
$\ul{(3)\Rightarrow (4)}$ follows from \prref{3.3}(2).\\
$\ul{(4)\Rightarrow (1)}$. It follows from \thref{3.5}(3) that there exist maps $\gamma_{xy}:\ G_{xy}\to A_{yx}\ot_{B_x} A_{xy}$
satisfying (\ref{eq:3.5.1}-\ref{eq:3.5.2}). Take $\sigma\in G_{xy}$. \equref{3.5.1} can be restated as
$$\sum_i\sum_{\tau\in G_{xy}} l_i(\sigma)r_i(\sigma)_\tau \ot \tau=1_y\ot \tau\in \oplus_{\tau\in G_{xy}} A_{yy}\ot k\tau.$$
Taking the projection of both sides to the component $A_{yy}\ot k\sigma$, it follows that
$$\sum_i l_i(\sigma)r_i(\sigma)_\sigma\ot\sigma = 1_y\ot \sigma,$$
$$1_y= \sum_i l_i(\sigma)r_i(\sigma)_\sigma= \sum_i \sum_{\tau\in G_{yx}}l_i(\sigma)_{\tau}r_i(\sigma)_\sigma
\in A_{yy}=\oplus_{\rho\in G_{yy}}A_\rho.$$
Taking the homogeneous components of degree $e_y$ of both sides, we find that
$$\sum_i l_i(\sigma)_{\sigma^{-1}}r_i(\sigma)_\sigma=1_y\in A_{\sigma^{-1}}A_{\sigma}.$$
This proves that $A_{\sigma^{-1}}A_{\sigma}=B_y$. Finally take $\tau\in G_{zx}$ and
$a\in A_{\tau\sigma}\subset A_{zy}$. Then
$$a=a1_y=\sum_i al_i(\sigma)_{\sigma^{-1}}r_i(\sigma)_\sigma\in A_{\tau\sigma\sigma^{-1}}A_\sigma=A_{\tau}A_\sigma.$$
Thus $A_{\tau\sigma}=A_\tau A_\sigma$ and $A$ is strongly graded.
\end{proof}

\begin{remark}\relabel{8.3}
Galois theory for finite groups acting on commutative extensions was introduced in \cite{AG}, see also
\cite{DI,KO} for an elegant presentation.
It was already observed by Chase and Sweedler \cite{CS} that these Galois extensions appear as Hopf-Galois
extensions over the Hopf algebra $H=(kG)^*$, the dual of the group algebra $kG$. One may also consider
Hopf-Galois extensions over the group algebra $kG$ itself, and these are precisely strongly graded algebras,
an observation that was first made by Ulbrich in \cite{U}. \thref{8.2} is the proper generalization of Ulbrich's
result. What is currently missing is a clear link to the classical theory, involving actions by groupoids, which would
make the picture complete. It is true that we have a theory involving actions, see Sections \ref{se:6} and \ref{se:7},
but this does not bring us what we would expect, since it involves actions by {\bf dual} $k$-linear categories,
while groupoids are ordinary $k$-linear categories. However, a Galois theory for groupoids acting (even partially)
on algebras was developed recently in \cite{BP,PT}. The connection to our theory seems unclear, our plan is to
investigate this in the future.
\end{remark}

\end{document}